\documentclass[12pt]{article}
\usepackage{amsfonts}
\usepackage{amssymb}

\usepackage{amsthm}
\usepackage{amsmath}

\newcommand{\n}{\mathfrak{n}}
\newcommand{\m}{\mathfrak{m}}

\newtheorem{defi}{\bf Definition}[section]
\newtheorem{theo}[defi]{\bf Theorem}
\newtheorem{lemma}[defi]{\bf Lemma}
\newtheorem{coro}[defi]{\bf Corollary}
\newtheorem{pro}[defi]{\bf Proposition}
\newtheorem{lem}[defi]{\bf Lemma}
\newtheorem{ex}[defi]{\bf Example}

\newtheorem{rem}[defi]{\bf Remark}

\title{Free nilpotent and nilpotent quadratic Lie algebras}

\author{P. Benito,  D. de-la-Concepci\'on, J. Laliena}

\date{\quad}

\begin{document}
\maketitle\vspace{-1.5cm}

\begin{abstract}
In this paper we introduce an equivalence between the category of the $t$-nilpotent quadratic Lie algebras with $d$ generators and the category of some symmetric invariant bilinear forms over the $t$-nilpotent free Lie algebra with $d$ generators. Taking into account this equivalence, $t$-nilpotent quadratic Lie algebras with $d$ generators are classified (up to isometric isomorphisms, and over any field of characteristic zero), in the following cases: $d=2$ and $t\leq 5$, $d=3$ and $t\leq 3$.
\end{abstract}

{\parindent= 2,5em \small  \sl Keywords: Nilpotent Lie algebras, 
invariant nondegenerate symmetric

bilinear forms, free nilpotent Lie algebras. }

{\parindent=10em \small  \sl }

\bigskip

{\parindent=2,5em \small \bf Classification MSC 2010: 17B01, 17B30}

\section{Introduction}

Let  ${\n}$ be a Lie algebra over an arbitrary field $\mathbb{K}$ of characteristic zero. The Lie algebra $\n$ is said to be nilpotent if $\n^{t+1}=0$, where $\n ^t$ is defined inductively as $\n^1 =\n$, $\n^i=[ \n^{i-1}, \n]$. In this case we call $t$ the index of nilpotency of $\n$ and we say that $\n$ is t-nilpotent or also t-step nilpotent ($\n^t\neq 0$). The chain of ideals of $\n$:
\begin{equation}\label{lcs}
\n \supseteq \n^2 \supseteq \dots \supseteq \n^t \supseteq \n^{t+1} \supseteq \dots
\end{equation}
is the well-known lower central series of $\n$. Hence, if $\n$ is t-nilpotent the lower central series finishes after $t+1$ steps.  
\medskip

The type of a nilpotent Lie algebra $\n$ is defined as the codimension of $\n ^2$ in $\n$. Following M. A. Gauger \cite[Section 1, Corollary 1.3]{G}, a set $\mathfrak m = \{ x_1, x_2, \dots, x_d \}$ generates $\n$ if and only if $\{ x_1+\n ^2,\dots x_d+\n ^2\}$ is a basis of $\n/\n ^2$. So, the type of a Lie algebra is the cardinal of every $\mathbb{K}$-linearly independent set,  $\mathfrak m = \{ x_1, x_2, \dots x_d\}$,  such that $\mathfrak t$,  the subspace generated by $\mathfrak m$, satisfies  $\mathfrak t \oplus \n^2 = \n$. The above conditions imply that $\mathfrak m = \{ x_1, \dots, x_d\}$ generates $\n$ as $\mathbb{K}$-algebra and therefore, we can see the elements $x_i\in \mathfrak m$ as a \emph{minimal set of generators} of $\n$.
\medskip

Let  $\mathfrak F \mathfrak L (d)$ be the free Lie algebra on a set of $d$ generators. The free $t$-nilpotent Lie algebra on $d$ generators is denoted $\n _{d,t}$ and defined as the quotient algebra
\begin{equation}\label{feedefinicion}
\n_{d,t}= \mathfrak F \mathfrak L (d) / \mathfrak F \mathfrak L (d)^{t+1}
\end{equation}
Any $t$-nilpotent Lie algebra $\n$ of type $d$ is a homomorphic image of $\n_{d,t}$. According to \cite[Section 1, Propositions 1.4 and 1.5]{G}, $\n \cong \n_{d,t}/I$, with $I$ an ideal of $\n_{d,t}$ such that $I\subseteq\n_{d,t}^2$ and $\n_{d,t}^t\not\subseteq I$.
\medskip

The Lie algebra, $\n$ is said to be quadratic (also known as metric) if it is endowed with a nondegenerate symmetric bilinear form $B$ which is invariant, that is,
\begin{equation}
B([x,y],z)= B(x,[y,z])
\end{equation}
The class of quadratic Lie algebras is related to physics and Riemannian geometry (see M. Bordemann \cite[Section 1]{Bo} for a explicit description of several connections). 
\medskip

Any semisimple Lie algebra is quadratic by means of the Killing form, $B(x,y)=Tr(ad\ x\circ ad\ y)$. In fact, the non degeneration of this trace form characterizes (in characteristic zero) the class of semisimple Lie algebras according to Cartan's Criterion. In the early 1980s, by the independent work of several authors (Kac, Favre and Santharouban, Medina and Revoy, Hofmann and Keith),
a recursive method to construct quadratic Lie algebras known as the \emph{double extension method} was developed (see \cite[Theorem 2.2]{Bo}). Starting with an abelian Lie algebra of dimension $0$ or $1$, and using one-dimensional central extensions and skew symmetric derivations, the method let us construct every finite-dimensional quadratic solvable Lie algebra. But this method seems to be difficult in high dimensions due to the multistep procedure.
\medskip

In recent years, we can find different research papers on the structure of nilpotent quadratic Lie algebras and also about partial classifications of them. For instance, the complete list of Lie algebras in this class has been given by G. Favre and I.J. Santharoubane \cite{F-S} in 1987 up to dimension $7$ and in 2007 by I. Kath \cite{K} up to dimension $10$ (only over the real field). The classification of solvable quadratic Lie algebras over algebraically closed fields given by M.T. Duong and R. Ushirobira \cite{D-U} in 2014 includes the complete list of nilpotent quadratic of dimension $8$. In \cite {O}, G. Ovando showed that there are $2$-step nilpotent quadratic Lie algebras of any type $d\geq 3$ with the exception of $d=4$; this fact was previously pointed out by Tsou and Walker \cite{T-W} in 1957. Following L. Noui and Ph. Revoy, \cite{N-R}, the classification of $2$-step nilpotent quadratic Lie algebras is reduced to the classification of alternating trilinear forms. In 2012, V. del Barco and G. Ovando \cite {B-O} proved that $\n_{3,2} $ and $\n_{2,3}$ are the unique free nilpotent quadratic Lie algebras using elementary properties of quadratic algebras. The classifications we have mentioned are mainly based on the double extension method. In this work we develop a general classification scheme for quadratic nilpotent Lie algebras based on the use of invariant bilinear forms on free nilpotent Lie algebras.
\medskip

The paper is divided into 5 sections from number $2$. Section 2 is devoted to give some examples of quadratic Lie algebras and examples of invariant bilinear forms on free nilpotent Lie algebras $\n_{d,t}$ of small dimensions. The group ${\rm Aut}\, \n_{d,t}$ of automorphisms of $\n_{d,t}$ is described in Section 3. A categorical approach between free nilpotent Lie algebras and their invariant bilinear forms and quadratic nilpotent Lie algebras is explained in Section 4. This approach let us identify the isometric isomorphism classes of quadratic $t$-nilpotent Lie algebras of type $d$ with the union of spaces of orbits of the natural action of the group ${\rm Aut}\, \n_{d,t}$ over some sets of invariant bilinear forms. The results given in Section 4 are used in Section 5  to classify up to isometric isomorphisms all the $t$-nilpotent quadratic Lie algebras in characteristic zero, with $d$ generators for $d=2$ and $t\leq 5$, $d=3$ and $t\leq 3$. The final Section includes the complete list of these algebras (basis and associated invariant bilinear form) over an algebraically closed field. Apart from the abelian $1$-dimensional, there are only six indecomposable quadratic algebras. The list includes the classification of the nilpotent quadratic Lie algebras up to dimension 7 and most of nilpotent quadratic Lie algebras of dimension 8. The section ends with the complete list of these algebras over the real field.
\bigskip

Throughout the paper $\mathbb{K}$ will denote an arbitrary field of characteristic zero except where otherwise specified. 
\section{Preliminaries and examples}

Let $\n_{d,t}$ be the free $t$-nilpotent Lie algebra on the set of generators $\m=\{x_1,\dots,x_d\}$. If $\mathfrak{s}_ i$ is the vector subspace generated by $\m^i$, and $\m^i$ denotes the set of all possible products made with $i$ elements of $\m$,  we have that
\begin{equation}
\n_{d,t}=\oplus _{k=1}^t \mathfrak{s}_k,
\end{equation}where $\oplus$ denotes the direct sum as vector spaces. This decomposition gives us a natural graduation on $\n_{d,t}$ that provides many structure results of free nilpotent Lie algebras in an easy way. Following M. Hall \cite{H}, we can get the well-known \emph{Hall basis of $\n_{d,t}$} whose elements are monomials in the generators. From M. Grayson and R. Grossman \cite[Definition 1.1]{G-G} each element in the Hall basis is defined recursively as:
\begin{enumerate}
\item $x_1, \dots, x_d$ are (ordered) elements of the basis of length 1.
\item The elements of lengths $1,\dots , r-1$ are simply ordered so that $a<b$ if $\text{length}(a)<\text{length}(b)$.
\item If $\text{length}(a)=s$ and $\text{length}(b)=t$ and $r=s+t$, $[a,b]$ is a basis element if: $a,b$ are basis elements and $a>b$; and in case $a=[c,d]$, then $b\geq d$. 
\end{enumerate} From the Hall basis, inductively we get the dimension of any subspace $\mathfrak{s}_l$, $l=1,\dots, t$ as
\begin{equation}\label{fdimensionsl}
\frac{1}{l}\sum_{a\mid l}\mu(a)d^{l/a},
\end{equation}where $\mu$ is the M\"oebius function.

\begin{ex}\label{Hallbasis} For $d=2,3$ and $t\leq 5$, the Hall basis (ordered), $\mathcal{H}_{d,t}$, of $\n_d,t$ are given as:
\begin{itemize}
\item $\mathcal{H}_{2,2}=\{x_1,x_2, [x_2,x_1]\}$,
\item $\mathcal{H}_{2,3}=\mathcal{H}_{2,2}\cup \{[[x_2,x_1],x_1],[[x_2,x_1],x_2]\}$;
\item $\mathcal{H}_{2,4}=\mathcal{H}_{2,3}\cup \{[[[x_2,x_1],x_1],x_1],[[[x_2,x_1],x_1],x_2], [[[x_2,x_1],x_2],x_2]\}$;
\item $\mathcal{H}_{2,5}=\mathcal{H}_{2,4}\cup \{[[[[x_2,x_1],x_1],x_1],x_1], [[[[x_2,x_1],x_1],x_1],x_2],[[[x_2,x_1],x_1],[x_2,x_1]],$ $[[[[x_2,x_1],x_1],x_2],x_2],[[[x_2,x_1],x_2],[x_2,x_1]], [[[[x_2,x_1],x_2],x_2],x_2]\}$; 
\item $\mathcal{H}_{3,2}=\{x_1,x_2,x_3, [x_2,x_1], [x_3,x_1], [x_3,x_2]\}$;
\item $\mathcal{H}_{3,3}=\mathcal{H}_{3,2}\cup \{[[x_2,x_1],x_j], [[x_3,x_1],x_j],[[x_3,x_2],x_k]: j=1,2,3, k=2,3\}$.
\end{itemize}
\end{ex}
For every Lie algebra $\mathfrak{g}$ we can define recursively the upper central series:
\begin{equation*}
Z_1(\mathfrak{g})=0 \quad \text{and}\quad 
Z_{i}(\mathfrak{g})=\{x\in \mathfrak{g}: [x,\mathfrak{g}]\subseteq Z_{i-1}(\mathfrak{g})\}.
\end{equation*}
Note that the center of $\mathfrak{g}$, $Z(\mathfrak{g})=\{x\in \mathfrak{g}: [x,\mathfrak{g}]=0\}$ is just $Z_2(\mathfrak{g})$.

In the case that $(\mathfrak{g},B)$ be quadratic (then $B$ is a nondegenerate symmetric invariant bilinear form), we have the following relation of orthogonality (see A. Meedina and Ph. Revoy \cite[Proposition 1.2]{M-R}) between the ideals in the lower central series of $\mathfrak{g}$ defined in (\ref{lcs}) and those of its upper central series ($\mathfrak{a}^\perp=\{x\in \mathfrak{g}:B(x,\mathfrak{a})=0\}$):
\begin{equation}
(\mathfrak{g}^i)^\perp=Z_{i}(\mathfrak{g})
\end{equation}
Hence, in any Lie algebra equipped with an invariant nondegenerate symmetric bilinear form the next equality holds:
\begin{equation}\label{fdimension}
\text{dim}\, \mathfrak{g}=\text{dim}\, \mathfrak{g}^i+\text{dim}\, Z_{i}(\mathfrak{g})
\end{equation}
For nilpotent Lie algebras, both series have the same length and this length determines the nilpotency index  of the algebra. In fact, $\mathfrak{g}$ is $t$-nilpotent if and only if $\mathfrak{g}^{t+1}=0$ if and only if $Z_{t+1}(\mathfrak{g})=\mathfrak{g}$. The upper central series of the free nilpotent Lie algebra $\n_{d,t}$ is:
\begin{equation}
Z_{i}(\n_{d,t})=\oplus_{k\geq t+2-i}\  \n_{d,t}^k\,.
\end{equation}

\begin{ex}\label{qabeliana} Any abelian Lie algebra $\n$ endowed with a nondegenerate bilinear form $B$ is a quadratic Lie algebra. Moreover, taking a fixed basis $\{x_1,\dots,x_d\}$ for $\n$, two quadratic Lie algebras $(\n,B)$ and $(\n,B')$ are (isometric) isomorphic if and only if the matrices $A_B=(B(x_i,x_j))$ and $A_{B'}=(B'(x_i,x_j))$ are congruents, i.e.: $A_{B'}=P^tA_BP$ for some regular matrix $P$.\end{ex}

\begin{ex} From \emph{(\ref{fdimension})}, in any quadratic Lie algebra $\mathfrak{g}$ the following equality holds $\text{dim}\, Z(\mathfrak{g})=\text{dim}\, \mathfrak{g}-\text{dim}\, \mathfrak{g}^2$. So, filiform algebras are not quadratic and the unique quadratic quasifiliform Lie algebra is $\n_{2,3}$. Generalized Heissenberg algebras are not quadratic (the general definition for GH is $\mathfrak{h}_n$ such that $\mathfrak{h}_n^2=Z(\mathfrak{h}_n)=\mathbb{K}\cdot z$).
\end{ex}
The following example summarizes part of the results on quadratic free nilpotent Lie algebras therein V. del Barco and G. Ovando \cite{B-O}:

\begin{ex} The center of $\n_{d,t}$ is exactly, $Z(\n_{d,t})=\n_{d,t}^t$. According to formulas \emph{(\ref{fdimensionsl})} and \emph{(\ref{fdimension})}, the dimension of $Z(\n_{d,t})$ is equal to $d_Z(d,t)=\frac{1}{t}\sum_{a\mid t}\mu(a)d^{t/a}$. From \emph{\cite[Proposition 3.7]{B-O}}, $2\leq d<d_Z(d,t)$ in case $t\geq 4$. On the other hand, $d_Z(d,1)=d$, $d_Z(d,2)=\frac{d(d-1)}{2}$ and $d_Z(d,3)=\frac{d(d^2-1)}{3}$. So, the only possible free nilpotent Lie algebras are $\n_{d,1}$, $\n_{3,2}$ and $\n_{2,3}$. From \emph{Example \ref{qabeliana}}, $n_{d,1}$ is quadratic; several nonisomorphic possibilities can occur depending on the base field. So are the Lie algebras  $\n_{2,3}$ and $\n_{3,2}$. In fact from \emph{G. Favre and L.J. Santharouban \cite{F-S}}, over algebraically closed fields of characteristic zero, any quadratic Lie algebra $(\n_{2,3}, B)$ or $(\n_{3,2}, B')$ is isometric to (the matrices are given in the basis $\mathcal{H}_{3,2}$ and $\mathcal{H}_{2,3}$):
\begin{equation}\label{2332invariant}
(\n_{3,2}, \varphi_{3,2}): \begin{pmatrix} 0 & 0 & 0 & 0 & 0 & 1 \\ 0 & 0 & 0 & 0 & -1 & 0 \\ 0 & 0 & 0 & 1 & 0 & 0 \\ 0 & 0 & 1 & 0 & 0 & 0  \\ 0 & -1 & 0 & 0 & 0 & 0 \\ 1 & 0 & 0 & 0 & 0 & 0 \end{pmatrix},\qquad (\n_{2,3}, \varphi_{2,3}): \begin{pmatrix} 0 & 0 & 0 & 0 & 1 \\ 0 & 0 & 0 & -1 & 0 \\ 0 & 0 & 1 & 0 & 0 \\ 0 & -1 & 0 & 0 & 0 \\ 1 & 0 & 0 & 0 & 0 \end{pmatrix}
\end{equation}
The results in Section 5 of this paper show that $(\n_{3,2}, \varphi_{3,2})$ is, up to isometries, the unique quadratic $2$-nilpotent Lie algebra of type $3$ over any field of characteristic zero. Over the real field, $(\n_{2,3}, \varphi_{2,3})$ is one of the two nonisometric quadratic $3$-nilpotent Lie algebra of type $2$. The other one is $(\n_{2,3},- \varphi_{2,3})$
\end{ex}

The next examples provide the whole vector space of invariant symmetric bilinear forms of $\n_{2,t}$ and $\n_{3,t}$ for small nilpotent index. The matrices are given in the Hall basis $\mathcal{H}_{d,t}$; they follow from a straightforward computation by using the definition of invariant bilinear form, and the Jacobi identity (this method is used in the proof of Theorem 3.8 in \cite{B-O} for $\n_{3,2}$).

\begin{ex}\label{formasinvariantespequenas} Any invariant symmetric bilinear form on $\n_{2,t}$ for $t\leq 5$ is of the form:

{\small \begin{equation}
B_{2,1}^{A_1}=\begin{pmatrix} \alpha & \beta \\ \beta & \delta \end{pmatrix}, B_{2,2}^{A_1}=\left( \begin{array}{cc|c} \alpha &\beta & 0 \\ \beta& \delta & 0 \\ \hline 0& 0& 0\end{array}\right), B_{2,3}^{A_1;\gamma}=\left( \begin{array}{cc|c|cc} \alpha &\beta & 0 & 0 & \gamma \\ \beta&\delta& 0 & -\gamma & 0 \\ \hline 0& 0 & \gamma & 0 & 0 \\ \hline 0 & -\gamma & 0 & 0 & 0 \\ \gamma & 0 & 0 & 0 & 0 \end{array}\right)
   \end{equation}}
   {\small \begin{equation}  
     B_{2,4}^{A_1;\gamma}= \left(\begin{array}{cc|c|cc|ccc} \alpha & \beta & 0 & 0 & \gamma & 0 & 0 & 0 \\  \beta & \delta & 0 & -\gamma & 0 & 0 & 0 & 0 \\ \hline 0 & 0 & \gamma & 0 & 0 & 0 & 0 & 0 \\  \hline 0 & -\gamma & 0 & 0 & 0 & 0 & 0 & 0 \\  \gamma &  0 & 0 & 0 & 0 & 0 & 0 & 0\\ \hline 0 & 0 & 0& 0 & 0 & 0 & 0 & 0 \\ 0 & 0 & 0 & 0 & 0 & 0 & 0 & 0 \\ 0 & 0 & 0 & 0 & 0 & 0 & 0 & 0 \end{array}\right)
   \end{equation}
}
   {\small   \begin{equation}
  B_{2,5}^{A_1;\gamma;A_2}= \left( \begin{array} {rr|r|rr|rrr|rrrrrr} \alpha & \beta & 0 & 0 & \gamma & 0 & 0 & 0 & 0 & -d &d & -e & e & -f\\ \beta & \delta & 0 & -\gamma& 0 & 0 & 0  & 0 & d & 0 &  e & 0 & f & 0 \\ \hline 0 & 0 & \gamma & 0 & 0 & -d & -e  & -f & 0 & 0 & 0 & 0  & 0 & 0 \\ \hline 0 & -\gamma & 0 & d & e & 0 & 0 & 0 & 0 & 0 & 0 & 0 & 0 & 0 \\ \gamma &  0 & 0 & e & f & 0 & 0 & 0 & 0 & 0 & 0 & 0 & 0 & 0\\ \hline 0 & 0 & -d & 0 & 0 & 0 & 0 & 0 & 0 & 0 & 0 & 0 & 0 & 0  \\  0 & 0 & -e & 0 & 0 & 0 & 0 & 0 & 0 & 0  & 0 & 0 & 0 & 0 \\ 0 & 0 & -f& 0 & 0 & 0& 0 & 0 & 0 & 0 & 0 & 0 & 0 &   0 \\\hline 0 & d & 0 & 0 & 0 & 0 & 0 & 0 & 0 & 0 & 0 & 0 & 0 & 0  \\ -d& 0 & 0 & 0 & 0 & 0  & 0 & 0 & 0 & 0 & 0 & 0 & 0 & 0 \\  d &  e & 0 & 0 & 0 &  0 & 0 & 0 & 0 & 0 & 0 & 0 & 0 & 0\\ -e& 0 & 0 & 0 & 0 & 0 & 0 & 0 & 0 & 0 & 0 & 0 & 0 & 0 \\ e & f & 0 & 0 & 0 & 0 & 0 & 0 & 0 & 0 & 0 & 0 & 0 & 0  \\ -f & 0 & 0 & 0 & 0 & 0 & 0 & 0 & 0 & 0 & 0 & 0 & 0 & 0    
\end{array} \right )
      \end{equation}}where $A_1=\begin{pmatrix} \alpha & \beta \\ \beta & \delta \end{pmatrix}$ and $A_2=\begin{pmatrix} d & e \\ e & f \end{pmatrix}$ are $2\times 2$ symmetric matrices and $\gamma\in \mathbb{K}$. This implies that the quadratic dimension of $\n_{2,1}, \n_{2,2}, \n_{2,3}, \n_{2,4}$ and $\n_{2,5}$ are $3, 3, 4, 4$ and $7$ respectively. 
\end{ex}

\begin{ex}\label{formasinvariantespequenas2}Any invariant symmetric bilinear form of $\n_{3,t}$ for $t\leq 3$ is of the form:
    {\small   \begin{equation}
  B_{3,1}^{A_1}=\begin{pmatrix} \alpha & \beta& \gamma \\ \beta & \delta &\epsilon\\\gamma & \epsilon & \omega \end{pmatrix}, B_{3,2}^{A_1;\lambda}=\left( \begin{array} {ccc|ccc}\alpha & \beta & \gamma & 0 & 0 & \lambda \\ \beta & \delta & \epsilon & 0 & -\lambda & 0 \\ \gamma & \epsilon & \omega & \lambda & 0 & 0 \\  	\hline 0 & 0 & \lambda & 0 & 0 & 0 \\ 0 & -\lambda & 0 & 0 & 0 & 0 \\ \lambda & 0 & 0 & 0 & 0 & 0 \end{array}\right)
 \end{equation}
 }
     {\small   \begin{equation}
  B_{3,3}^{A_1;\lambda;A_2}=\left( \begin{array} {rrr|rrr|rrrrrrrr} \alpha & \beta & \gamma & 0 & 0 & \lambda & 0 & a & b & 0 & b & d & c &  e \\ \beta & \delta & \epsilon & 0 & -\lambda& 0 & -a & 0 & c & -b & 0 & e & 0 & f \\ \gamma & \epsilon & \omega & \lambda & 0 & 0 & -b & -c & 0 & -d& -e & 0 & -f & 0 \\ \hline0 & 0 & \lambda & a & b & c & 0 & 0 & 0 & 0 & 0 & 0 & 0 & 0 \\0 & -\lambda & 0 & b & d& e & 0 & 0 & 0 & 0 & 0 & 0 & 0 & 0\\\lambda & 0 & 0 & c & e & f & 0 & 0 & 0 & 0 & 0 & 0 & 0& 0 \\\hline 0 & -a & -b & 0 & 0 & 0 & 0 & 0 & 0 & 0 & 0 & 0 & 0 & 0 \\a & 0 & -c & 0 & 0 & 0 & 0 & 0 & 0 & 0 & 0 & 0 & 0 & 0\\b & c & 0 & 0 & 0 & 0 & 0 & 0 & 0 & 0 & 0 & 0 & 0 & 0\\0 & -b & -d &  0 & 0 & 0 & 0 & 0 & 0 & 0 & 0 & 0 & 0 & 0\\b & 0 & -e&  0 & 0 & 0 & 0 & 0 & 0 & 0 & 0 & 0 & 0 & 0 \\d &e & 0 & 0 & 0 & 0 & 0 & 0 & 0 & 0 & 0 & 0 & 0 & 0\\ c & 0& -f & 0 & 0 & 0 & 0 & 0 & 0 & 0 & 0 & 0 & 0 & 0\\ e & f & 0 & 0 & 0 & 0 & 0 & 0 & 0 & 0 & 0 & 0 & 0 & 0 \end{array} \right ) 
  \end{equation}
  }where $A_1=\begin{pmatrix} \alpha & \beta& \gamma \\ \beta & \delta &\epsilon\\\gamma & \epsilon & \omega \end{pmatrix}$ and $A_2=\begin{pmatrix} a & b& c \\ b & d &e\\c & e & f \end{pmatrix}$ are $3\times 3$ symmetric matrices and $\lambda\in K$. So the quadratic dimension of $\n_{3,1}$, $\n_{3,2}$ and $\n_{3,3}$ are $6,7$ and $13$ respectively.
\end{ex}

\section{The group of automorphisms of $\n_{d,t}$}

For every $a\in \n_{d,t}$ we have a unique decomposition $a= a_1 + \dots + a_t$ where $a_i\in \mathfrak{s}_i$ and we can consider, for each $k=1, \dots, t$, the projection endomorphism of $\n_{d,t}$, which is a linear endomorphism:
\begin{eqnarray}\label{k-proyecciones}
\begin{split}
e_k \colon & \n_{d,t} &\to &\quad \n_{d,t} \\
&a &\mapsto &\quad e_k(a)= a_k
\end{split}
\end{eqnarray}
We notice that $e_k$ are idempotent endomorphisms of $\n_{d,t}$ and:
$$e_1 + \dots + e_t= \mathrm{I}_{\n_{d,t}},$$ where $\mathrm{I}_{\n_{d,t}}$ is the identity homomorphism of $\n_{d,t}$.
 
Following T. Sato \cite[Proposition 3]{S}, any linear map $\varphi:\mathfrak{s}_1\to\n_{d,t}$ may be extended to a (unique) endomorphism as an algebra $\Phi_\varphi:\n_{d,t}\to \n_{d,t}$ in the following way: for any monomial of length $k$, $a=[x_{i_1}\dots x_{i_k}]$, $\Phi_\varphi(a)=[\varphi(x_{i_1})\dots \varphi(x_{i_k})]$. Moreover $\Phi_\varphi$ is an automorphism if and only if $\{e_1(\varphi(x_1)), \dots e_1(\varphi(x_d))\}$ is a set of linearly independent elements of $\mathfrak{s}_1$ if and only if $e_1 \circ \varphi\in GL(\mathfrak{s}_1)$. In this way, $\Phi_{\text{I}_{\mathfrak{s}_1}}=\mathrm{I}_{\n_{d,t}}$ if $\text{I}_{\mathfrak{s}_1}:\mathfrak{s}_1\to\n_{d,t}$ is defined over the generator elements of $\n_{d,t}$ as $\text{I}_{\mathfrak{s}_1}(x_i)=x_i$.

\begin{pro}\label{productosemidirecto} 
The group $\mathrm{Aut}\, \n_{d,t}$ decomposes as a semidirect product of the subgroups $\text H(d,t)=\{\Phi_\varphi: \varphi \in GL(\mathfrak{s}_1)\}$ and $N(d,t)=\{\Phi_\rho: \rho=\mathrm{I}_{\mathfrak{s}_1}+\sigma, \sigma \in \mathrm{End}\,(\mathfrak{s}_1, \n_{d,t}^2)\}$. Even more, $N(d,t)$ is a normal subgroup which is nilpotent  and $H(d,t)$ is isomorphic to the group of regular $d\times d$ matrices with entries in $\mathbb{K}$.
\end{pro} 
	\begin{proof}Firstly we note that the elements of $H(d,t)$ are graded automorphisms and $H(d,t)\cap N(d,t)=1$. For $\varphi_i \in GL(\mathfrak{s}_1)$, we have $\Phi_{\varphi_1}(\Phi_{\varphi_2})^{-1}=\Phi_{\varphi_1}\Phi_{\varphi_2^{-1}}=\Phi_{\varphi_1\varphi_2^{-1}}$, so $H(d,t)$ is a subgroup. Analogously, given $\rho_i=\mathrm{I}_{\mathfrak{s}_1}+\sigma_i$ where $\sigma_i \in \text{End}\,(\mathfrak{s}_1, \n_{d,t}^2)$, $\Phi_{\rho_2}^{-1}=\Phi_\rho$ where $\rho=\mathrm{I}_{\mathfrak{s}_1}-{\Phi_{\rho_2}}^{-1}\sigma_2$. In addition, $\Phi_{\rho_1}\Phi_{\rho_2}=\Phi_\rho$ with $\rho=\mathrm{I}_{\mathfrak{s}_1}+\sigma_1+\Phi_{\rho_1}\sigma_2$ and $\Phi_{\varphi_1}\Phi_{\rho_1}(\Phi_{\varphi_1})^{-1}=\Phi_\rho$ where $\rho=\mathrm{I}_{\mathfrak{s}_1}+\Phi_{\varphi_1}\sigma_1\varphi_1^{-1}$. Hence, $N(d,t)$ is a subgroup and $H(d,t)$ is contained in $N_{\text{Aut}\, \n_{d,t}}(N(d,t))$, the normaliser of $N(d,t)$ in $\text{Aut}\, \n_{d,t}$. Now every $\varphi:\mathfrak{s}_1\to\n_{d,t}$ that satisfies $e_1\circ \varphi\in GL(\mathfrak{s}_1)$ decomposes in the form:
	$$
	\Phi_\varphi=\Phi_{e_1\varphi}\Phi_\rho, \rho=\mathrm{I}_{\mathfrak{s}_1}+e_2(\Phi_{e_1\varphi})^{-1}e_2\varphi+\dots e_t(\Phi_{e_1\varphi})^{-1}e_t\varphi.
	$$Hence $N(d,t)$ is a normal subgroup and $\text{Aut}\, \n_{d,t}$ is a semidirect product of $H(d,t)$ and $N(d,t)$. The last assertion is easily checked.\end{proof}
	 
For technical purposes, we give the following alternative description for the subgroups $H(d,t)$ and $N(d,t)$ in Proposition \ref{productosemidirecto}. If $\varphi \in \text{Aut}\, \n_{d,t}$, we have that $\varphi (\mathfrak{s}_i)\subseteq \oplus _{j=i}^t \mathfrak{s}_j$ for $i=1,\dots, t$.  Therefore  for $j<k$, $e_j \varphi e_k =0$. So
\begin{equation}
 \varphi = \sum_{j,k=1}^t e_j \varphi e_k= \sum_{t\geq j \geq jk  \geq 1} e_j \varphi e_k,
\end{equation}
and we have that
\begin{equation}\label{H(d,t)}
H(d,t)=\{ \sum_{i=1}^t e_i\varphi e_i \quad / \quad \varphi \in \text{Aut}\, \n_{d,t}\}
\end{equation}
\begin{equation}\label{N(d,t)}
N(d,t)= \{ \sum_{t\geq j> k \geq 1} e_j \varphi e_k + Id \quad / \quad \varphi \in \text{Aut}\,\n_{d,t}\}
\end{equation}

\noindent Using Proposition \ref{productosemidirecto} and the Hall basis $\mathcal{H}_{d,t}$ described in Example \ref{Hallbasis}, from straightforward computations we get the following examples:

\begin{ex}
If we consider in $\n_{2,3}$ the basis $\mathcal{H}_{2,3}$, the elements of $\mathrm{Aut}\,\n_{2,3}$ can be identified with matrices in $GL(5)$. In fact, the elements of $H(2,3)$ and $N(2,3)$are given by matrices
\begin{equation}
\left ( \begin{array} {c  | c |  c}
  A & 0 & 0  \\ \hline 0 & det(A) & 0 \\ \hline 0 & 0 & det(A) A
   \end{array}
    \right ) \quad \text{and} \quad \left ( \begin{array} {c c  | c | c c} 1 & 0 &  0 & 0 & 0 \\  0 & 1 & 0 & 0 & 0 \\ \hline \alpha & \beta & 1 & 0 & 0 \\ \hline \delta & \gamma & \beta & 1 & 0 \\ \mu & \epsilon & -\alpha & 0 &  1   \end{array} \right ),
\end{equation}
where $A\in GL(2)$, and $\alpha, \beta,  \delta, \gamma, \epsilon, \mu \in \mathbb{K}$. It is easy to check that $N(2,3)$ satisfies $N(2,3)^2=1$ (here $N(2,3)^2$ means $[N(2,3),[N(2,3),N(2,3)]]$ with $[N(2,3),N(2,3)]$ the commutator of two subgroups, i.e, $[P,Q]=\{[p,q]=p^{-1}q^{-1}pq: p\in P, q\in Q\}$).
\end{ex}

\begin{ex}
Consider now $\n_{3,2}$ and the basis $\mathcal{H}_{3,2}$. The matrices representing the elements of $\mathrm{Aut}\,\n_{3,3}$ in relation to $\mathcal{H}_{3,2}$ are given by 
$$\left ( \begin{array} { c | c } A & 0 \\ \hline B &  \begin{array}{ccc} A_{3,3} & A_{3,2} & A_{3,1} \\ A_{2,3} & A_{2,2} & A_{2,1} \\ A_{1,3} & A_{1,2} & A_{1,1}\end{array}  \end{array} \right )$$
where $A\in GL(3)$, $B\in Mat_{3,3}(\mathbb{K})$ (where $Mat_{3,3}(\mathbb{K})$ denotes the set of $3\times 3$ matrices with entries in $\mathbb{K}$), and the $A_{i,j}$ matrices are obtained deleting row $i$ and column $j$ in $A$ and taking the determinant of the resulting matrix. In this case $H(3,2)$ and $N(3,2)$ are given respectively by matrices of the types
\begin{equation}
\left ( \begin{array} { c | c } A & 0 \\ \hline 0 &  \begin{array}{ccc} A_{3,3} & A_{3,2} & A_{3,1} \\ A_{2,3} & A_{2,2} & A_{2,1} \\ A_{1,3} & A_{1,2} & A_{1,1}\end{array}  \end{array} \right )  \quad \text{and} \quad  \left ( \begin{array} { c | c } I_3 & 0 \\ \hline B &  I_3  \end{array} \right )
\end{equation}
where $I_3$ denotes the  $3\times 3$ identity matrix.
\end{ex}

\section{Free and quadratic nilpotent Lie algebras}

In this section we shall introduce a new technique of constructing quadratic nilpotent Lie algebras out of free nilpotent Lie algebras endowed with an invariant symmetric bilinear form.
\medskip 

\noindent 4.1. {\bf Categorical approach.} Following \cite{G}, any $t$-step nilpotent Lie algebra $\n$ of type $d$ is a homomorphic image of $\n_{d,t}$. In fact, given an arbitrary isomorphism of Lie algebras $ \varphi  \colon \n_{d,t}/ I \to \n$ the ideal $I$ satisfies that $\n_{d,t}^t\not\subseteq I\subseteq \n_{d,t}^2$. Indeed, if $x\in I$ and $x\notin \n_{d,t}^2$, then from the introduction,  $x$ can be seen as a generator of $\n_{d,t}$, but this is a contradiction because $\n_{d,t} / I \cong \n$ and $\n$ has $d$ generators. On the other hand, if $\n_{d,t}^t \subseteq I$, we have $( \n_{d,t} /I)^t =0$ and then $\n^t=0$, a contradiction with the $t$-nilpotency of $\n$.

Let $(\n, B)$ be a $t$-nilpotent quadratic Lie algebra of type $d$ and $ \varphi  \colon \n_{d,t}/ I \to \n$ be any isomorphism of Lie algebras. We can define on $\n _{d,t}$ the following symmetric bilinear form:
\begin{equation}\label{invarianteinducida}
B_1 (x, y) = B(\varphi (x+ I), \varphi(y+ I)).
\end{equation}
Using that $\varphi$ is isomorphism and that $B$ is an invariant form on $\n$, we easily check that $B_1$  is an invariant form:
\begin{eqnarray*}
B_1([x,y],z)&= &B(\varphi( [x+I, y+I]), \varphi (z + I)] = B[(\varphi (x+I), \varphi (y+I)], \varphi (z+I))\\ &=& B(\varphi (x+I), [\varphi (y+I), \varphi(z+I)]) =  B_1(x,[y,z]).
 \end{eqnarray*}
We also note that if $ B_1(x,y) =0$ for every $y\in \n_{d,t}$, then $0= B(\varphi(x+I), z)$ for every $z\in \n$ because  $\varphi$ is surjective. Since $B$ is nondegenerate we get that $x+I =0$, that is $Ker (B_1)= I$. Hence $\n_{d,t}^t\not\subseteq Ker (B_1) \subseteq \n_{d,t}^2$.
 
Suppose now that $\n_{d,t}$ is endowed with an invariant symmetric bilinear form $U$ and consider the orthogonal subspace $\n_{d,t}^	\perp=Ker (U)$. For every $x\in Ker (U)$ and $y, z \in \n_{d,t}$
$$U([x,y],z)=U(x,[y,z])=0.$$So, $[Ker (U),\n_{d,t}]\subseteq Ker (U)$ which proves that $Ker (U)$ is an ideal. Then we can define on the Lie algebra $\n_{d,t}/ Ker (U)$ the invariant symmetric nondegenerate bilinear form
\begin{equation}\label{invariantededucida}
\overline {U}(x+Ker(U), y+Ker(U)) = U(x,y).
\end{equation}Following the previous approach, if we start with a $t$-nilpotent quadratic Lie algebra $(\n, B)$ of type $d$ and $\varphi \colon \n_{d,t}/I  \to \n$ is an isomorphism of Lie algebras and $B_1$ is the invariant symmetric bilinear form defined in (\ref{invarianteinducida}), then we have that $(\n_{d,t}/ I, \overline {B_1})$ is a nilpotent quadratic Lie algebra. Even more, $\varphi$ is an isomorphism of Lie algebras and an isometry from $(\n_{d,t}/I, \overline {B_1})$ onto $(\n,B)$, because
$$ \overline{B_1} (x+I, y+I) =  B_1(x,y)= B(\varphi (x+I), \varphi(y+I)).$$
 
Previous discussion can be settled in the following result:

\begin{pro}\label{previosafunctor}
Let $(\n, B)$ be  a quadratic $t$-nilpotent Lie algebra of type $d$ and $\varphi \colon \n_{d,t}/I \to \n$ be an isomorphism of Lie algebras. Then:
\begin{enumerate}
\item[{\rm (i)}] The map $B_1 \colon \n_{d,t} \times \n_{d,t} \to K$  given by $ B_1 (x, y) = B(\varphi(x+I), \varphi(y+I))$ is an invariant symmetric bilinear form on $\n_{d,t}$.
\item[{\rm (ii)}] The orthogonal subspace of $B_1$ is exactly $\n_{d,t}^\perp= Ker (B_1) = I$ and satisfies that $\n_{d,t}^t\nsubseteq I\subseteq \n_{d,t}^2$.
\item[{\rm (iii)}] The map $\overline{B_1} \colon \n_{d,t}/I \times \n_{d,t}/I \to K$ defined as $\overline {B_1}(x+I, y+I)=B_1(x,y)$ is an invariant nondegenerate symmetric bilinear form on $\n_{d,t}/I$.
\item[{\rm (iv)}] $\varphi $ is an isometry from $(\n_{d,t}/I, \overline {B_1})$ onto $(\n, B)$.
\end{enumerate}
\end{pro}

Let us define $\bf {NilpQuad_{d,t}} $ the category whose objects are the $t$-nilpotent quadratic Lie algebras $(\n, B)$ of type $d$, and whose morphisms are Lie homomorphisms 
$$\varphi \colon (\n, B) \to (\n^\prime, B^\prime)$$
 such that $B(x,y) = B^\prime (\varphi (x), \varphi(y))$. We will call these Lie homomorphisms, \emph{metric Lie homomorphisms}.

We define also $\bf Sym_0(d,t)$ the category whose objects are the symmetric invariant bilinear forms $B$ on the free Lie algebra $\n_{d,t}$ for which $Ker (B) \subseteq \n_{d,t} ^2$ and $\n_{d,t}^t\nsubseteq Ker (B)$, and whose morphisms are defined as follows: for any $B_1, B_2\in Obj(\bf Sym_0(d,t))$, we introduce the set of metric Lie endomorphisms of $\n_{d,t}$ that respect the kernel of the bilinear form, i.e.:
\begin{align*}
MEnd_\perp (B_1, B_2):=\qquad\qquad\qquad\qquad\qquad\qquad\qquad\qquad\qquad\qquad\qquad\qquad\qquad\\
\{ \varphi \in End(\n_{d,t}) :  B_1(x,y) = B_2(\varphi (x), \varphi(y)), \varphi (Ker (B_1))\subseteq Ker (B_2)\}.
 \end{align*}
The whole set of morphisms from $B_1$ to $B_2$ is defined as the quotient set
\begin{equation}
Hom(B_1, B_2):= MEnd_\perp (B_1, B_2) /\sim
\end{equation}
where $\sim$ is the equivalence relation,
\begin{equation}\label{equivalencia}
\varphi_1\sim \varphi_2 \Longleftrightarrow (\varphi_1-\varphi_2)(\n_{d,t})\subseteq Ker(B_2) \quad \forall_{\varphi_1, \varphi_2\in MEnd_\perp (B_1, B_2))}.
\end{equation}
 
The morphisms of $\bf Sym_0(d,t)$ are well-defined: if $\varphi  \in MEnd_\perp (B_1, B_2)$ and $\psi \in MEnd_\perp (B_2, B_3)$ for $B_1, B_2, B_3 \in Obj (\bf Sym_0(d,t))$, $\psi \circ \varphi (Ker(B_1)) \subseteq (Ker(B_3))$, so $\psi\circ \varphi \in MEnd_\perp (B_1, B_3)$. Moreover, if $[\varphi]=[\varphi']$ and $[\psi]=[\psi']$, since $(\varphi-\varphi')(\n_{d,t})\subseteq Ker(B_2)$ and $(\psi-\psi')(\n_{d,t})\subseteq Ker(B_3)$ and
$$
\psi\varphi-\psi'\varphi'=\psi(\varphi-\varphi')-(\psi'-\psi)\varphi',
$$we get $(\psi\varphi-\psi'\varphi')(\n_{d,t})\subseteq Ker(B_3)$. Hence, $[\psi\varphi]=[\psi'\varphi']$
 
Now we define the functor
\begin{equation}\label{functor}
\bf {Q_{d,t}} \colon{ \bf Sym_0(d,t)} \to {\bf NilpQuad_{d,t}}
\end{equation}
such that for every $B\in  Obj ({\bf Sym_0(d,t)})$
\begin{equation*}
{\bf Q_{d,t}}(B) = (\n_{d,t}/Ker (B), \overline {B}),
\end{equation*}
where  $\overline {B} \colon \n_{d,t}/Ker(B)\times \n_{d,t}/Ker(B) \to K$ is defined by $\overline{B} (x+Ker(B), y+Ker(B))= B(x,y)$. And, for every morphism $[\varphi] \in Hom (B_1, B_2)$,
 \begin{eqnarray}
 \begin{split}
 {\bf Q_{d,t}} ([\varphi])   \colon & (\n_{d,t}/Ker(B_1), \overline {B_1})  & \to &\quad  (\n_{d,t}/ Ker (B_2), \overline {B_2})\\ & \quad x+ Ker (B_1) & \mapsto & \quad \varphi (x) + Ker(B_2).
 \end{split}
 \end{eqnarray}
 We remark that $\bf Q_{d,t}$ is a well-defined functor for objects, because of Proposition \ref{previosafunctor}. Also $\bf Q_{d,t}$ is well-defined for morphisms. If $\varphi, \psi \in MEnd_\perp (B_1, B_2)$ such that $[\varphi] = [\psi]$, then 
$${\bf Q_{d,t}}([\varphi])(x + Ker(B_1))=\varphi (x) + Ker (B_2)$$
 and 
 $${\bf Q_{d,t}}([\psi])(x + Ker(B_1))=\psi (x) + Ker (B_2)$$
 for every $x\in \n_{d,t}.$
  But, from (\ref{equivalencia}), $(\varphi - \psi)(\n_{d,t})\subseteq Ker (B_2)$, and therefore,  ${\bf Q_{d,t}}([\varphi])(x + Ker(B_1))= {\bf Q_{d,t}}([\psi])(x + Ker(B_1))$.
\medskip

The following theorem stablishes that the functor $\bf Q_{d,t}$ defined in (\ref{functor}) provides an equivalence between both categories, $\bf Sym_0(d,t)$ and $\bf NilpQuad_{d,t}$.

\begin{theo}\label{eqcategorias}
The categories $\bf Sym _0(d,t)$ and $\bf NilpQuad_{d,t}$ are equivalent.
\end{theo}

\begin{proof}
We will show that the functor $\bf Q_{d,t}$ is faithful and full, and that for every object $(\n, B)$ in $\bf NilpQuad_{d,t}$ there exists an object $U$ in $\bf Sym_0(d,t)$ such that $ \bf Q_{d,t} (U)$ and $\bf (\n, B)$ are isomorphic in  $\bf NilpQuad_{d,t}$ (that is, the functor is essentially surjective or dense). So, according to \cite [Proposition 1.3]{J}, we will have that both categories are equivalent.

From Proposition \ref{previosafunctor} (iv), we get that $\bf Q_{d,t}$ is dense. Now we prove that $\bf Q_{d,t}$ is a faithful functor. Indeed, if ${\bf Q_{d,t}}([\varphi]) = {\bf Q_{d,t}}([\psi])$, with $[\varphi], [\psi] \in Hom (B_1, B_2)$ and $B_1, B_2 \in Obj ({\bf Sym_0 (d,t)})$ then $\varphi(x) + Ker (B_2)= \psi (x) + Ker (B_2) $ for all $x\in \n_{d,t}$. That is,  $(\varphi - \psi)(\n_{d,t})\subseteq Ker(B_2)$ and therefore $ [\varphi]=[\psi]$ by using (\ref{equivalencia}).

Finally we check that $\bf Q_{d,t}$ is a full functor. Let $\tau$ be any element in the set of morphisms $Hom ({\bf Q_{d,t}}(B_1), {\bf Q_{d,t}}(B_2))$. For a fixed $\{x_1, \dots, x_d\}$ set of generators of $\n_{d,t}$, take $y_i \in \n_{d,t}$ such that $y_i + Ker(B_2)= \tau(x_i + Ker (B_1))$  for $i= 1, \dots, d$. From the universal mapping property of $\n_{d,t}$  the correspondence $x_i\to y_i$ extends uniquely to a Lie homormorphism $\varphi \colon \n_{d,t} \to \n_{d,t}$. Since $\varphi(x_i) + Ker (B_2)= y_i + Ker (B_2)= \tau (x_i + Ker (B_1))$ and since the elements $x_i$ generate $\n_{d,t}$, we have $\tau\pi_1=\pi_2 \varphi$ where $\pi_i:\n_{d,t}\to \n_{d,t}/Ker(B_i)$ is the canonical surjection. This implies  $\varphi(Ker (B_1) )\subseteq Ker (B_2)$ and $\varphi(x) + Ker (B_2) = \tau (x + Ker (B_1))$, for all  $x \in \n_{d,t}$.  Moreover, $\varphi$ is a metric endomorphism:
 \begin{eqnarray*}
 B_2(\varphi(x), \varphi(y)) &=& \overline {B_2}(\varphi (x) + Ker (B_2), \varphi (y) + Ker (B_2))= \\ & & \overline {B_2} (\tau (x + Ker (B_1), \tau (y + Ker (B_1))=\\
 & & \overline {B_1} ( x + Ker (B_1), y + Ker (B_1))= B_1(x,y).
 \end{eqnarray*}
Therefore $[\varphi] \in Hom (B_1, B_2)$ and ${\bf Q_{d,t}} ([\varphi]) = \tau$.
\end{proof}

As a consequence of this equivalence of categories we have:
\begin{coro}\label{isomorfismo}
For all $B_1, B_2 \in Obj ({\bf Sym _0(d,t)})$, the following assertions are equivalent:
\begin{enumerate}
\item[{\rm (i)}] $B_1$ and  $ B_2$ are isomorphic in ${\bf Sym _0(d,t)}$.
\item[{\rm (ii)}] ${\bf Q_{d,t}}(B_1)$ and ${\bf Q_{d,t}}(B_2)$ are isometrically isomorphic Lie algebras.
\item[{\rm (iii)}] There exists a metric automorphism $\theta:(\n_{d,t},B_1)\to (\n_{d,t},B_2)$.
\end{enumerate}
\end{coro}

\begin{proof} The equivalence ${\rm (i)}\Longleftrightarrow{\rm (ii)}$ follows from Theorem \ref{eqcategorias} and ${\rm (ii)}\Longleftrightarrow{\rm (iii)}$ follows from Proposition 1.6 in \cite{G} and the proof therein.
\end{proof}

Since ${\bf Sym_0(d,t)}$ and ${\bf NilpQuad}_{d,t}$ are equivalent categories, there is a bijection between the isomorphism types of objects of each category. So,  the classification of t-nilpotent quadratic Lie algebras with $d$  generators up to isometric isomorphisms is the classification of objects  in the category ${\bf NilpQuad}_{d,t}$ up to isomorphism, and this classification, in turn, is the one of objects in the category $\bf Sym_0(d,t)$ up to isomorphism.
\medskip

\noindent 4.2. {\bf The group $\text{Aut}\, \n_{d,t}$ acting on ${\bf Sym_0(d,t)}$.} The group of automorphisms of $\n_{d,t}$ acts on the set $Obj ({\bf Sym _0(d,t)})$ in the natural way:
\begin{equation}\label{accion}
\mathrm{Aut}\, \n_{d,t}\times Obj ({\bf Sym _0(d,t)})\to Obj ({\bf Sym _0(d,t)}), (\theta, B)\to B_\theta
\end{equation}
where $B_\theta (x,y)=B(\theta(x),\theta(y))$. Indeed, for every $\theta\in Aut\, \n_{d,t}$, $\theta: (\n_{d,t}, B_\theta)\to (\n_{d,t}, B)$ is a metric Lie isomorphism. Even more:

\begin{lem}\label{isomorfismos} For all $B\in Obj ({\bf Sym _0(d,t)})$, the set $Orb_{\mathrm{Aut}\, \n_{d,t}}(B)=\{B_\theta: \theta\in \mathrm{Aut}\, \n_{d,t}\}$ is equal to the set of bilinear invariant symmetric forms isomorphic to $B$ in the category ${\bf Sym _0(d,t)}$. Therefore the number of orbits of the action described in \emph{(\ref{accion})} is exactly the number of isomorphism types in the classification of $t$-nilpotent quadratic Lie algebras of type $d$ up to isometries.
\end{lem}

\begin{proof}
For the first assertion, we remember that $B$ and $B'$ are isomorphic objects in ${\bf Sym _0(d,t)}$ if and only if there exist $\varphi\in MEnd_\perp (B, B')$ and $\tau \in MEnd_\perp (B', B)$ such that $[\varphi \tau]=[\tau\varphi]=[\text{I}_{\n_{d,t}}]$. It is clear that $B$ is isomorphic to $B_\theta$ for every $\theta\in \text{Aut}\, \n_{n,t}$. Conversely, if $B'$ is isomorphic to $B$, from Corollary \ref{isomorfismo}, there exists a metric automorphism $\theta:(\n_{d,t}, B')\to(\n_{d,t}, B)$. So, $B(\theta(x),\theta(x))=B'(x,y)$ and therefore, $B'=B_\theta\in Orb_{\mathrm{Aut}\, \n_{d,t}}(B)$.
\end{proof}

From now on, we will denote as $S_0^2(d,t)$ the $\mathbb{K}$-vector space
\begin{equation}\label{S02}
S_0^2(d,t)=\{ B  : B \text { is an invariant symmetric bilinear form on } \n_{d,t}\}.
\end{equation}
{\noindent We note that the elements in $Obj(\bf Sym_0(d,t))$ are not closed by the sum of bilinear forms. If  $B_1, B_2 \in Obj(\bf Sym_0(d,t))$, it  can happen that $Ker (B_1 + B_2)\nsubseteq \n_{d,t}^2$, and also that $\n_{d,t}^t\subseteq Ker(B_1+B_2)$. On these occasions we can see $B_1 + B_2$ as a invariant bilinear form on $\n_{d_1,t}$ with $1\leq d_1\leq d$, or on $\n_{d,t_1}$, with $1\leq t_1\leq  t$, respectively. 

\noindent Also we notice that  it could be interesting to consider the more general action:
 \begin{eqnarray*}
 \text{Aut}\,\n_{d,t} \times S_0^2(d,t)  &\rightarrow & S_0^2(d,t)\\
 (\theta, B) & \mapsto & B_\theta
 \end{eqnarray*}
We observe that if $B\in S_0^2(d,t)$,  and we define
\begin{eqnarray}
\begin{split}
 B(e_i,e_j): \n_{d,t} \times \n_{d,t}  &\rightarrow & \mathbb{K} \qquad \qquad\\
 (x,y) & \mapsto & B(e_i(x), e_j(y))
 \end{split}
 \end{eqnarray}
then
$$B= \sum_{i,j=1}^t B(e_i, e_j),$$
where $e_k$ are the projections defined in (\ref{k-proyecciones}). We also observe that  $B(e_i, e_j)=0$ if $i+j>t+1$. Defining for $k=1, \dots, t$,
\begin{eqnarray}\label{Bk-componentes}
B_k= \sum_{i=1}^{t-k+1} B(e_i, e_{t-i-k+2}),
 \end{eqnarray}
we have the decomposition
$$B= \sum_{k=1}^t B_k,$$
where $B_k$ for $k=1, \dots, t$ are invariant symmetric bilinear forms, because:
\begin{eqnarray*}
B_k(a,[b,c])&=& \sum_{i=1}^{t-k+1} B(e_i(a), e_{t-i-k+2}([b,c]))= \\& &\sum_{i=1}^{t-k+1} B(e_i(a), \sum_{l+m=t-i-k+2}[e_l(b),e_m(c)]) =\\& &\sum_{i=1}^{t-k+1} \sum_{l+m=t-i-k+2}B(e_i(a), [e_l(b),e_m(c)]) =\\& & \sum_{m=1}^{t-k+1} \sum_{l+i=t-m-k+2} B([e_i(a), e_l(b)], e_m(c)) = \\& & \sum_{m=1}^{t-k+1} B(e_{t-m-k+2}([a,b]), e_m(c)) = B_k ([a,b],c)
\end{eqnarray*}
\begin{ex}
In \emph{Example \ref{formasinvariantespequenas}}, the sum decomposition $B_{2,3}^{A;\gamma}= B_1+B_2+B_3$ will be (matrices are given in the Hall basis $\mathcal {H}_{2,3}$)
$$\left( \begin{array}{cc|c|cc} 0 &0 & 0 & 0 & \gamma \\ 0& 0 & 0 & -\gamma & 0 \\ \hline 0 & 0 & \gamma & 0 & 0 \\ \hline 0 & -\gamma & 0 & 0 & 0 \\ \gamma & 0 & 0 & 0 & 0 \end{array}\right) + \left( \begin{array} {cc|c|cc}0 & 0 & 0 & 0 & 0 \\0 & 0& 0 & 0 & 0 \\ \hline 0 & 0 & 0 & 0 & 0 \\ \hline 0 & 0 & 0 & 0 & 0 \\ 0 & 0 & 0 & 0 & 0 \end{array}\right) + \left( \begin{array}{cc|c|cc} \alpha & \beta & 0 & 0 & 0 \\ \beta & \delta & 0 & 0 & 0 \\ \hline 0 & 0 & 0 & 0 & 0 \\ \hline 0 & 0& 0 & 0 & 0 \\ 0 & 0 & 0 & 0 & 0 \end{array}\right).$$
\end{ex}

The $k$-components related to invariant bilinear symmetric that we have introduced in (\ref{Bk-componentes}) have special features: 

\begin{pro}\label{Bk}
For any  $B\in S_0^2(d,t)$, the following assertions holds: 
\begin{enumerate}
\item[{\rm (i)}]$B_k= \sum_{i=1}^{t-k+1} B(e_i, e_{t-i-k+2})$ for $k=1, \dots, t$ are invariant symmetric bilinear forms.
\item[{\rm (ii)}]In the sum $ B_k= \sum_{i=1}^{t-k+1} B(e_i, e_{t-i-k+2})$, any fixed addend $B(e_i, e_{t-i-k+2})$, determines completely the whole set of  addends. 
\item[{\rm (iii)}]If $B\in Obj({\bf Sym_0(d,t)})$ then $B_1\not=0$. In fact, $B(e_i,e_j)\not=0$ for $i+j= t+1$.
\end{enumerate}
\end{pro}

\begin{proof}
Item (i) has been already proved and item (ii) is a consequence of the invariance of $B$. To show (iii), we take into account that $\n_{d,t}^t \nsubseteq Ker(B)$ if $B\in Obj({\bf Sym_0(d,t)})$. So, as $B(e_i,e_t)=0$ for $i=2, \dots, t$, $B(e_1, e_t)\not =0$; otherwise, ${\mathfrak s}_t=\n_{d,t}^t\subseteq Ker(B)$, a contradiction. So $B_1\not=0$, and from (ii) we get $B(e_i,e_j)\not=0$ for $i,j =1, \dots, t$ such that $i+j=t+1$.
\end{proof}

Next, we consider the set of $B_1$-components of every invariant symmetric bilinear form $B$ in $S_0^2(d,t)$, that is:
$$S_{00}^2 (d,t)= \{ \sum_{i=1}^t B(e_i, e_{t-i+1}) : B \in S_0^2(d,t)\}.$$
We note that $S_{00}^2(d,t)$ is a $\mathbb{K}$-vector subspace of invariant symmetric bilinear forms of $S_0^2(d,t)$. Moreover, if $B\in S_{00}^2(d,t)$ and $\varphi \in H(d,t)$, from (\ref{H(d,t)}) we have
\begin{eqnarray*}
B_{\varphi} &=&\sum_{i=1}^t B_{\varphi}(e_i, e_{t-i+1})= \sum_{i=1}^t B(e_i \sum_{j=1}^t e_j\varphi e_j, e_{t-i+1} \sum_{j=1}^t e_j\varphi e_j) =\\& & \sum_{i=1}^t B(e_i\varphi e_i, e_{t-i+1}\varphi e_{t-i+1})\in S_{00}^2(d,t)
\end{eqnarray*}
So, we can conclude that $S_{00}^2(d,t)$ is a $H(d,t)$-module. In the case that $\varphi \in N(d,t)$, using (\ref{N(d,t)}) we get
\begin{eqnarray*}
B_{\varphi} &=&\sum_{i=1}^t B_{\varphi}(e_i, e_{t-i+1})= \sum_{i=1}^t B(e_i (Id+ \sum_{j>k}e_j\varphi e_k), e_{t-i+1} (Id+\sum_{j>k} e_j\varphi e_k)) =\\& &   \sum_{i=1}^t B(\sum_{i>k}e_i\varphi e_k, e_{t-i+1}) + \sum_{i=1}^t B(e_i, \sum_{t-i+1>k} e_{t-i+1}\varphi e_k) + \\ & & \sum_{i=1}^t B(\sum_{i>k}e_i\varphi e_k, \sum_{t-i+1>k} e_{t-i+1}\varphi e_k) +B
\end{eqnarray*}
\noindent  and therefore $(B_\varphi)_1=B$. That is, the action on $B\in S_{00}^2(d,t)$ of the elements of $N(d,t)$ provides invariant bilinear forms that have $B$ as 1-component. Summarizing:

\begin{lemma}
For any $B\in S_{00}^2(d,t)$, $B_\varphi \in S_{00}^2(d,t)$ for every $\varphi \in H(d,t)$ and if $\varphi \in N(d,t)$ then $(B_\varphi)_1=B$.
\end{lemma}

From $S_{00}^2(d,t)$ it is possible to obtain a large amount of elements in $Obj ({\bf Sym_0(d,t)})$.  In fact, for small $d$ and $t$, we will prove in the next section that the orbits of $S_{00}^2(d,t)$ provide the whole set of elements in ${\bf Sym_0(d,t)}$.

\section{Small type quadratic nilpotent Lie algebras}

A quadratic Lie algebra $(L,\varphi)$ is called \emph{decomposable} if it contains a proper ideal $I$ such that $\varphi_{I\times I}$ is nondegenerate, otherwise $L$ is called indecomposable. In this case, $L=I\oplus I^\perp$, $\varphi_{I^\perp\times I^\perp}$ is nondegenerate and $I^\perp$ is an ideal. Hence, any quadratic Lie algebra is the direct sum of orthogonal indecomposable Lie algebras. If $L=I_1\oplus I_2$ is a decomposable quadratic nilpotent Lie algebra, the type of $L$ is the sum of the types of $I_1$ and $I_2$ and the nilpotent index of $L$ is equal to ${\rm max}\{t_1,t_2\}$, where $t_i$ is the nilpotent index of $I_i$. Then, any quadratic abelian Lie algebra of dimension $\geq 2$ is decomposable and the quadratic nonabelian nilpotent  Lie algebras of type 2 are indecomposable.

We recall that two matrices, $A$ and $B$, are congruent if there exists a regular matrix $P$ such that $B=P^tAP$.

\begin{theo} Over any arbitrary field of characteristic $0$, the objects in the category ${\bf Sym_0(d,1)}$, $d\geq 1$, are the nondegenerate symmetric  bilinear forms $B_{2,1}^{A}$ determined by a $d\times d$ symmetric and regular matrix $A$.  Moreover:
\begin{enumerate}
\item[{\rm (i)}]Two bilinear forms $B_{d,1}^{A_1}$ and $B_{d,1}^{A_2}$ are isomorphic in ${\bf Sym_0(d,1)}$ if and only if $A_1$ and $A_2$ are congruent.
\item[{\rm (ii)}]Any quadratic Lie algebra of the form ${\bf Q_{d,1}}(B_{d,1}^{A})=(\mathfrak{n}_{d,1}, B_{d,1}^{A})$ decomposes as an orthogonal direct sum of $1$-dimensional ideals which are quadratic Lie algebras.
\item[{\rm (iii)}]$(\mathfrak{n}_{d,1}, B_{d,1}^{A_1})$ and $(\mathfrak{n}_{d,1}, B_{d,1}^{A_2})$ are isomorphic as metric Lie algebras if and only if the matrices $A_1$ and $A_2$ are congruent.
\end{enumerate}
\end{theo}

\begin{proof} The kernel of any bilinear invariant form $B\in {\bf Sym_0(d,1)}$ is contained in $\mathfrak{n}_{2,1}^2=0$. So, these forms are of maximal rank $d$. The assertion (i) follows from the fact that $\text{Aut}\,\n_{d,1} = GL(d)$. Then the bilinear forms in the same orbit as $B_{2,1}^{A_1}$ are $B_{2,1}^{P^t A_1P}$, with $P\in GL(d)$. Finally, (ii) follows from the fact that any bilinear form on a finite dimensional vector space (characteristic $\neq 2$) has an orthogonal basis and (iii) follows from Corolllary \ref{isomorfismo}. 
\end{proof}

\begin{theo}\label{type2theorem} Over any field $\mathbb{K}$ of characteristic $0$, the set $Obj({\bf Sym_0(2,t)})$ is empty if $t=2,4$. For $t=3,5$ we have the following:
\begin{enumerate}

\item[{\rm (i)}]The objects in ${\bf Sym_0(2,3)}$ are the invariant forms $B_{2,3}^{A;\gamma}$ described in \emph{Example \ref{formasinvariantespequenas}} and given by any $2\times 2$ symmetric matrix $A$ and any $0\neq \gamma\in \mathbb{K}$. Moreover, $B_{2,3}^{A_1;\gamma}$ and $B_{2,3}^{A_2;\delta}$ are isomorphic in ${\bf Sym_0(2,3)}$ if and only if $\displaystyle{\frac{\delta}{\gamma}\in \mathbb{K}^2}$.

\item[{\rm (ii)}]The objects in ${\bf Sym_0(2,5)}$ are the invariant forms $B_{2,5}^{A_1;\gamma;A_2}$ described in \emph{Example \ref{formasinvariantespequenas}} for any $\gamma\in \mathbb{K}$ and any $2\times 2$ symmetric matrices $A_1$ and $A_2$ with ${\rm rank}\, A_2\geq 1$. Moreover, $B_{2,5}^{A_1;\gamma;A_2}$ and $B_{2,5}^{B_1;\delta;B_2}$ are isomorphic in ${\bf Sym_0(2,5)}$ if and only if $B_2=({\rm det}\, P)^2P^tA_2P$ for some $2\times 2$ regular matrix $P$.
\end{enumerate}

\end{theo}

\begin{proof} According to Proposition \ref{Bk}, any $B\in Obj ({\bf Sym_0(2,t)})$ satisfies that $0\neq B_1$. We will prove that for $t=2,4$ any invariant bilinear form $B$ satisfies $B_1=0$. The invariant form $B_1$ is completely determined by $B(e_i,e_j)$, as defined in (29),  for any pair $(i,j)$ such that $i+j=t+1$. If $t=2$, we consider the pair $(i,j)=(1,2)$ and note that: ${\mathfrak s}_1=span<x_1,x_2>$ and ${\mathfrak s}_2=span<[x_2,x_1]>$. From the invariance of $B$, we have that $B(x_2,[x_2,x_1])=B([x_2,x_2],x_1])=0$ and $B(x_1,[x_2,x_1])=-B([x_1,x_1],x_2])=0$. For $t=4$, we take the pair $(1,4)$; in this case, ${\mathfrak s}_4=span<[[[x_r,x_s],x_p],x_q]: (r,s)=(2,1), (1,2), p,q\in \{1,2\}>$ and $B(x_r, [[[x_r,x_s],x_p],x_r])=0$ and $B(x_r, [[[x_r,x_s],x_p],x_s])=-B([x_r,x_s],[[x_r,x_s],x_p])=0$. Hence, for $t=2,4$ the set $Obj({\bf Sym_0(2,t)})$ is empty because of (iii) in Proposition \ref{Bk}.

It is clear that $\mathfrak{n}_{2,3}^3$ is not contained in the kernel of $B_{2,3}^{A;\gamma}$ if and only if $\gamma\neq 0$. In this case $B_{2,3}^{A;\gamma}$ is nondegenerate and $(\mathfrak{n}_{2,3},B_{2,3}^{A;\gamma})$ is a quadratic Lie algebra.
Now, let $B_{2,3}^{{\tiny \left(
\begin{array}{cc}
u& v\\
v& w
\end{array}
\right)}
,\gamma}, B_{2,3}^{{\tiny \left(
\begin{array}{cc}
0& 0\\
0& 0
\end{array}
\right)}
,\gamma}\in Obj ({\bf Sym_0(2,3)})$ and consider $\theta_P\in \text{Aut}\,\n_{2,3}$ the automorphism given by the matrix (in the Hall basis $\mathcal{H}_{2,3}$):$$P=\left( \begin{array} {cc|c|cc}1 & 0 & 0 & 0 & 0 \\ 0 & 1 & 0 & 0 & 0\\\hline 0& 0 & 1 & 0 & 0 \\ \hline -\frac{1}{\gamma}v & -\frac{1}{2\gamma}w & 0 & 1 & 0 \\ \frac{1}{2\gamma}u & 0 & 0 & 0 & 1\end{array}\right).
$$We have that$$
P^t\left( \begin{array} {cc|c|cc}0 & 0 & 0 & 0 & \gamma \\ 0 & 0 & 0 & -\gamma & 0\\\hline 0& 0 & \gamma & 0 & 0 \\ \hline 0 & -\gamma & 0 & 0 & 0 \\ \gamma & 0 & 0 & 0 & 0\end{array}\right)P=\left( \begin{array} {cc|c|cc}u & v & 0 & 0 & \gamma \\ v & w & 0 & -\gamma & 0\\\hline 0& 0 & \gamma & 0 & 0 \\ \hline 0& -\gamma& 0 & 0 & 0 \\ \gamma & 0 & 0 & 0 & 0\end{array}\right)
$$On the other hand, for $\gamma, \delta \in \mathbb{K}$ such that $\gamma\delta\neq 0$ and $\frac{\delta}{\gamma}=\epsilon^2$ with $\epsilon \in \mathbb{K}$, the automorphism $\theta_Q\in \text{Aut}\,\n_{2,3}$ given by the matrix:$$Q=\left( \begin{array} {cc|c|cc}\epsilon & 0 & 0 & 0 & 0 \\ 0 & 1 & 0 & 0 & 0\\\hline 0& 0 & \epsilon & 0 & 0 \\ \hline 0& 0& 0 & \epsilon^2 & 0 \\ 0& 0 & 0 & 0 & \epsilon\end{array}\right)$$satisfies:
$$Q^t\left( \begin{array} {cc|c|cc}0 & 0 & 0 & 0 & \gamma \\ 0 & 0 & 0 & -\gamma & 0\\\hline 0& 0 & \gamma & 0 & 0 \\ \hline 0 & -\gamma & 0 & 0 & 0 \\ \gamma & 0 & 0 & 0 & 0\end{array}\right)Q=\left( \begin{array} {cc|c|cc}0 & 0 & 0 & 0 & \delta \\ 0 & 0 & 0 & -\delta & 0\\\hline 0& 0 & \delta & 0 & 0 \\ \hline 0& -\delta& 0 & 0 & 0 \\ \delta & 0 & 0 & 0 & 0\end{array}\right).$$This implies from Lemma \ref{isomorfismos} that $B_{2,3}^{{\tiny \left(
\begin{array}{cc}
u& v\\
v& w
\end{array}
\right)}
,\gamma}$ and $B_{2,3}^{{\tiny \left(
\begin{array}{cc}
0& 0\\
0& 0
\end{array}
\right)}
,\gamma}$ are isomorphic in the category ${\bf Sym_0(2,3)}$ and so are  $B_{2,3}^{{\tiny \left(
\begin{array}{cc}
0& 0\\
0& 0
\end{array}
\right)}
,\gamma}$ and $B_{2,3}^{{\tiny \left(
\begin{array}{cc}
0& 0\\
0& 0
\end{array}
\right)}
,\delta}$ in case $\gamma\delta\neq 0$ and $\frac{\delta}{\gamma}=\epsilon^2$ with $\epsilon \in \mathbb{K}$.

Finally, let as assume that $B_{2,3}^{A_1;\gamma}$ and $B_{2,3}^{A_2;\delta}$ are isomorphic in ${\bf Sym_0(2,3)}$. Then, from previous paragraph, $B_{2,3}^{{\bf 0}_{2\times 2};\gamma}$ and $B_{2,3}^{{\bf 0}_{2\times 2};\delta}$ are also isomorphic and because of Lemma \ref{isomorfismos}, we can find $\theta_R\in \text{Aut}\,\n_{2,3}$ where ($\epsilon=ad-bc\neq 0$)$$R=\left ( \begin{array} {c c  | c | c c} a & b &  0 & 0 & 0 \\  c & d & 0 & 0 & 0 \\ \hline \alpha_1 & \alpha_2 & \epsilon & 0 & 0 \\ \hline \alpha_3 & \alpha_4 & \alpha_2 & \epsilon a & \epsilon b \\ \alpha_5 & \alpha_6 & -\alpha_1 & \epsilon c &  \epsilon d   \end{array} \right)
$$such that $R^tB_{2,3}^{{\bf 0}_{2\times 2};\gamma}R=B_{2,3}^{{\bf 0}_{2\times 2};\delta}$. The last matrix equation implies $\gamma \epsilon^2=\delta$ and the equivalence in (i) follows.

In the sequel we consider the objects in the category ${\bf Sym_0(2,5)}$. For any invariant bilinear form $B_{2,5}^{A_1;\gamma;A_2}$ on  $\mathfrak{n}_{2,5}$, it is easily checked that $\mathfrak{n}_{2,5}^3$ is not contained in the kernel of this form if and only if ${\rm rank}\, A_2\geq 1$ if and only if the kernel is contained in $\mathfrak{n}_{2,5}^2$. Let $B_{2,5}^{A_1;\gamma;A_2}, B_{2,5}^{{\bf 0}_{2\times 2};0;A_2}\in Obj({\bf Sym_0(2,5)})$ where$$A_1=\left(
\begin{array}{cc}
p& q\\
r& s
\end{array}
\right), A_2=\left(
\begin{array}{cc}
d& e\\
e& f
\end{array}\right)
\neq {\bf 0}_{2\times 2}.$$We are interested in proving that both objects are isomorphic. For this, we will use the automophism subgroup $N(2,5)$. The matrix elements in this subgroup (with respect to the Hall basis $\mathcal{H}_{2,5}$) are of the form:$$X(A,C,D,E)= \left( \begin{array} {c|c|c|c|c} I_2 & 0 & 0 & 0& 0 \\ \hline A& 1 & 0 & 0 &0 \\ \hline C&  A^\prime & I_2 & 0 & 0 \\ \hline D & C^\prime & A^{\prime \prime}& I_3 & 0 \\   \hline E &  D^\prime& C^{\prime \prime}& A^{\prime \prime \prime} & I_6    \end{array} \right ),$$
where $A\in M_{1\times 2}(K), C\in M_{2\times 2}(K), D\in M_{3\times 2}(K), E\in M_{6\times 2}(K)$,  but $A^\prime \in M_{2\times 1}(K), A^{\prime \prime} \in M_{3\times 2}(K), A^{\prime\prime \prime} \in M_{6\times 3}(K)$ are determined by $A$, and $C^\prime \in M_{3\times 1}(K),$ $C^{\prime \prime} \in M_{6\times 2}(K)$ are determined by $C$ and $D^\prime \in M_{6\times 1}(K)$ is determined by $D$. Let us assume firstly that $f\neq 0$ and consider the matrices
$$A={\bf 0}_{1\times 2}, C=\left(
\begin{array}{cc}
0& 0\\
\frac{\gamma}{2f}& 0
\end{array}\right), D={\bf 0}_{3\times 2},E^t=\left(
\begin{array}{cccccc}
0&0&0&0&0&\frac{\gamma^2-4pf}{8f^2}\\
0&0&0&0&\frac{r}{2f}&\frac{er-2qf}{2f^2}
\end{array}\right)$$For the automorphism $\theta_{X(A,C,D,E)}$ we have that $A',A'',A'''$ and $D'$ are zero matrices and $(C')^t=(0\ 0\ -\frac{\gamma}{2f})$ and $$X(A,C,D,E)^tB_{2,5}^{{\bf 0}_{2\times 2};0;A_2}X(A,C,D,E)=B_{2,5}^{A_1;\gamma;A_2}.
$$(We note that, denoting by $B_{\theta_X}$ the invariant bilinear form on the left part and using Proposition \ref{Bk}, for checking the identity, we only need to compute $B_{\theta_X}(e_i,e_i)$ for $i=1,2,3$; this computations involves the matrices $C'$, $A_2, A'_2,A''_2$ and $E$.)

Next let us suppose that $f=0, d\neq 0$ and consider the matrices
$$A={\bf 0}_{1\times 2}, C=\left(
\begin{array}{cc}
0& -\frac{\gamma}{2d}\\
0& 0
\end{array}\right), D={\bf 0}_{3\times 2},E^t=\left(
\begin{array}{cccccc}
\frac{q}{d}&-\frac{p}{2d}&0&0&0&0\\
\frac{4dr-\gamma^2}{8d^2}&0&0&0&0&0
\end{array}\right).$$In this case, $(C')^t=( -\frac{\gamma}{2d}\ 0\ 0)$ and $$X(A,C,D,E)^tB_{2,5}^{{\bf 0}_{2\times 2};0;A_2}X(A,C,D,E)=B_{2,5}^{A_1;\gamma;A_2}.
$$
Finally, if $f=d=0$, we have $e\neq 0$ and we can consider the automorphism $\theta(A,C,D,E)$ where$$A={\bf 0}_{1\times 2}, C=\left(
\begin{array}{cc}
\frac{\gamma}{2e}& 0\\
0& 0
\end{array}\right), D={\bf 0}_{3\times 2},E^t=\left(
\begin{array}{cccccc}
0&0&\frac{q}{e}&0&\frac{p}{2e}&0\\
0&0&\frac{r}{2e}&0&0&0
\end{array}\right).$$
Then, $(C')^t=( 0\ -\frac{\gamma}{2e}\ 0)$ and $$X(A,C,D,E)^tB_{2,5}^{{\bf 0}_{2\times 2};0;A_2}X(A,C,D,E)=B_{2,5}^{A_1;\gamma;A_2}.$$From the previous reasoning, the final assertion in item (ii) of this theorem is equivalent to the following statement: $B_{2,5}^{{\bf 0}_{2\times 2};0;A_2}$ and $B_{2,5}^{{\bf 0}_{2\times 2};0;B_2}$ are isomorphic if and only if $B_2=({\rm det}\, P)^2P^tA_2P$. The necessary condition follows from the general form of any automorphism of $\mathfrak{n}_{2,5}$:$$X(P;A,C,D,E)= \left( \begin{array} {c|c|c|c|c} P & 0 & 0 & 0& 0 \\ \hline A& P' & 0 & 0 &0 \\ \hline C&  A^\prime & P'' & 0 & 0 \\ \hline D & C^\prime & A^{\prime \prime}& P''' & 0 \\   \hline E &  D^\prime& C^{\prime \prime}& A^{\prime \prime \prime} & P''''    \end{array} \right ),$$
where $P$ is a regular matrix, $P'={\rm det}\, P$ and $P''={\rm det}\, P\cdot P$; the rest of the submatrices under the principal diagonal are as in the case $N(2,5)$. From the notion of isomorphism we get$$
X(P;A,C,D,E)^tB_{2,5}^{{\bf 0}_{2\times 2};0;A_2}X(P;A,C,D,E)=B_{2,5}^{{\bf 0}_{2\times 2};0;B_2}
$$and so we get the relation $B_2=(det\, P)^2P^tA_2P$. The converse follows by using the automorphism $\theta_{X(P;0,0,0,0)}$.

\end{proof}

If $\mathbb{K}$ is a field, we can consider the group $\mathbb{K}^*$ with the product of the field, where $\mathbb{K}^*= \mathbb{K}-\{0\}$. Then $(\mathbb{K}^*)^2= \{ k^2: k\in \mathbb{K^*}\}$ is a subgroup of $\mathbb{K}^*$, and we can consider the quotient group $\mathbb{K}^* /\mathbb({\mathbb{K}}^*)^2$. 

\begin{coro}\label{type2} Over any field $\mathbb{K}$ of characteristic $0$, there are no nilpotent quadratic Lie algebras of type $2$ and nilpotent index $2$ or $4$. In addition, up to metric isomorphism:
\begin{enumerate}

\item[{\rm (i)}]The nilpotent quadratic Lie algebras of type 2 and nilpotent index 3 are of the form $(\mathfrak{n}_{2,3},B_{2,3}^{0;\gamma})$ for any $0\neq \gamma\in \mathbb{K}$. Moreover, $(\mathfrak{n}_{2,3},B_{2,3}^{0;\gamma})$ and $(\mathfrak{n}_{2,3},B_{2,3}^{0;\delta})$ are isometrically isomorphic if and only if  $\delta$ and $\gamma$ are congruent module $({\mathbb {K}^*})^2$.

\item[{\rm (ii)}]The nilpotent quadratic Lie algebras of type 2 and nilpotent index 5 are of the form $(\frac{\mathfrak{n}_{2,5}}{{\rm Ker}\, B_{2,5}^{0;0;A}},\overline{B_{2,5}^{0;0;A}})$ for any arbitrary $2\times 2$ nonzero symmetric matrix $A$. Moreover, $(\frac{\mathfrak{n}_{2,5}}{{\rm Ker}\, B_{2,5}^{0;0;A}}, \overline{B_{2,5}^{0;0;A}})$ and $(\frac{\mathfrak{n}_{2,5}}{{\rm Ker}\, B_{2,5}^{0;0;B}}, \overline{B_{2,5}^{0;0;B}})$ are isometrically isomorphic  if and only if $B=({\rm det}\, P)^2P^tAP$ for some $P$ regular $2\times 2$ matrix.
\end{enumerate}

\noindent All the quadratic algebras described in {\rm (i)} and {\rm (ii)} are indecomposable. We notice that in (i) the number of nonisomorphic algebras is the cardinal of the group $\mathbb{K}^*/(\mathbb{K}^*)^2$. Also we notice that in (ii),  if  $\mathbb{K}$ satisfies that for every $\alpha \in  \mathbb{K}$ there exists $\delta \in \mathbb{K}$ such that $\delta ^3 = \alpha$, then $(\frac{\mathfrak{n}_{2,5}}{{\rm Ker}\, B_{2,5}^{0;0;A}}, \overline{B_{2,5}^{0;0;A}})$ and $(\frac{\mathfrak{n}_{2,5}}{{\rm Ker}\, B_{2,5}^{0;0;B}}, \overline{B_{2,5}^{0;0;B}})$ are isomorphic as metric Lie algebras if and only if $A$ and $B$ are congruent.

\end{coro}

\begin{proof} We just need to apply the previous Therorem \ref{type2theorem} and Corollary \ref{isomorfismo} to prove (i) and  (ii). For the last statement about (ii), we notice that if $B= P^tAP$ with $P$ regular, then $R= (\sqrt[3]{det(P)}) ^{-1} P$ satisfies that $B= det(R)^2 R^t A R$.

\end{proof}

From now on, we will tackle the case ${\bf Sym_0(3,t)}$. From Example \ref{formasinvariantespequenas2}, the invariant bilinear forms on $\mathfrak{n}_{3,2}$ and $\mathfrak{n}_{3,3}$ are given by the matrices on the Hall basis:
\begin{equation}\label{A1A2}
(\mathfrak{n}_{3,2}, B^{A_1;\gamma}): \left ( \begin{array} {c|c}A_1& W_\gamma\\ \hline W_\gamma & 0 \end{array} \right), \quad (\mathfrak{n}_{3,3},B^{A_1;\gamma;A_2}): \left ( \begin{array} {c|c|c}A_1& W_\gamma& A_2^\prime\\ \hline W_\gamma & A_2& 0 \\ \hline A_2^{\prime t} & 0 & 0\end{array} \right) 
\end{equation}
where $A_1$ and $A_2$ are $3\times 3$ symmetric matrices, $\delta \in \mathbb{K}$ and 
\begin{equation}\label{Wgamma-C}
W_\delta=\begin{pmatrix} 0 & 0 & \gamma  \\ 0 & -\gamma & 0 \\  \gamma & 0 & 0 \end{pmatrix}=\gamma \cdot C, \quad\mathrm{where}\quad C=\begin{pmatrix} 0 & 0 & 1 \\ 0 & -1 & 0 \\  1 & 0 & 0 \end{pmatrix};
\end{equation}\label{objetos 3-2-3}
$A^\prime_2$ is a $3\times 8$ matrix completely determined by $A_2$ as follows:
\begin{equation}\label{A2A2prima}
A_2={\small\left(
\begin{array}{ccc}
a& b& c\\
b& d&e\\
c& e&f
\end{array}
\right)}\quad \text{and}\quad 
A_2'=\begin{pmatrix} 0 & a & b & 0 & b & d & c & e \\ -a & 0 & c & -b & 0 & e & 0 & f \\  -b & -c & 0 & -d & -e & 0 & -f &0 \end{pmatrix}.
\end{equation}

In the following lemma, we discuss the indecomposibility of $\mathbf{Q_{3,t}}(B)$ for $B\in {\bf Sym_0(3,t)}$.
\medskip

\begin{lem}\label{restricciones} Let $\gamma$ be any scalar and $A_1$, $A_2$ symmetric $3\times 3$ matrices over a field of characteristic $0$. The objects in ${\bf Sym_0(3,2)}$ are the bilinear forms $B_{3,3}^{A_1;\gamma}$ with $\gamma\neq 0$ and for the bilinear forms $B_{3,3}^{A_1;\gamma;A_2}$, we have that:
\begin{enumerate}
\item[\emph{(i)}] ${\rm rank\, A_2}\geq 2$ if and only if ${\rm rank\, A_2'}= 3$. In this case, $B_{3,3}^{A_1;\gamma;A_2}$ is a well-defined object in ${\bf Sym_0(3,3)}$.
\item[\emph{(iii)}]${\rm rank\, A_2}= 1$ if and only if ${\rm rank\, A_2'}= 2$.
\end{enumerate}
Moreover, $\mathbf{Q_{3,3}}(B)$ is an indecomposable quadratic Lie algebra except for the objects in ${\bf Sym_0(3,3)}$ of the form $B_{3,3}^{A_1;\gamma;A_2}$ with ${\rm rank\, A_2}= 1$.
\end{lem}

\begin{proof} The first assertion is clear. To prove (i) and (ii), let $A_2$ and $A_2^\prime$ be as in (\ref{A2A2prima}). Denoting as $c_i$ the columns of $A_2^\prime$ and taking the matrix $D$ with columns $c_i(D)$ given by $c_1(D)=c_2, c_2(D)=c_5, c_3(D)= c_7, c_4(D)=-c_1, c_5(D)=-c_4, c_6(D)=c_3-c_5, c_7(D)= c_6, c_8(D)= c_8$ and $D'$ obtained form $D$ by reordering rows and columns:
\begin{equation*}
D=\begin{pmatrix} a & b & c & 0 & 0 & 0 & d & e \\ 0 & 0 & 0 & a & b & c & e & f \\  -c & {\footnotesize -}e & {\footnotesize -}f & b & d & e & 0 &0 \end{pmatrix},\quad D'=\begin{pmatrix} a & b & c & 0 & 0 & 0 & e & f \\ b & d & e & {\footnotesize -}c & {\footnotesize -}e & {\footnotesize -}f & 0 & 0 \\  0 & 0 & 0 & a & b & c & d &e \end{pmatrix}
\end{equation*}
We observe that ${\rm rank}\, A_2'={\rm rank}\, D={\rm rank}\, D'$. If either ${\rm rank}\, A_2=3$, or  ${\rm rank}\, A_2=2$ with $(a,b,c)\neq (0,0,0)$, it is clear that ${\rm rank}\, D'=3$. If ${\rm rank}\, A_2=2$ and  $(a,b,c)=(0,0,0)$, then $df\neq e^2$ and therefore ${\rm rank}\, D'=3$. Let us assume that ${\rm rank}\, A_2=1$. If $(a,b,c)=(0,0,0)$, using $df=e^2$ we easily get ${\rm rank}\, A_2'={\rm rank}\, D'=2$. Otherwise $(a,b,c)\neq(0,0,0)$ and $\lambda(a,b,c)=(b,d,e)$, $\mu(a,b,c)=(c,e,f)$. Then, the matrix $D''$ obtained from $D$ by changing the third row $r_3$ by $r'_3=r_3+\mu r_1-\lambda r_2$ and $r'_1=r_1$ and $r'_2=r_2$ is:
\begin{equation}
D''=\begin{pmatrix} a & b & c & 0 & 0 & 0 & d & e \\ 0 & 0 & 0 & a & b & c & e & f \\  0 & 0 & 0 & 0 & 0 & 0 & \mu d-\lambda e=0 &\mu e-\lambda f=0 \end{pmatrix}
\end{equation}
Hence, ${\rm rank}\, A_2'={\rm rank}\, D'={\rm rank}\, D''=2$. This proves (i) and (ii).

Finally, let $B\in {\bf Sym_0(3,t)}$ such that $\mathbf{Q_{3,t}}(B)=(\n, \overline{B})$ is a decomposable quadratic Lie algebra. Then, $\n=\frac{\n_{3,t}}{Ker(B)}$ is the orthogonal sum of  two quadratic Lie algebras of types $2$ and $1$. In this case, there exists a minimal generator set $\{x_1,x_2,x_3\}$ of $\n_{3,t}$ such that $[x_3,x_1], [x_3,x_2] \in Ker (B)$, so the rank of the form $B(e_2,e_2)$ is $1$. Therefore, $t=3$ and $B=B_{3,3}^{A_1;\gamma;A_2}$ with ${\rm rank\, A_2}= 1$. On the other hand, any quadratic Lie algebra $\mathbf{Q_{3,3}}(B)$ with $B=B_{3,3}^{A_1;\gamma;A_2}\in {\bf Sym_0(3,3)}$ and ${\rm rank\, A_2}= 1$ is of type $3$ and dimension $5$. Thus the derived algebra $\n^2$ has dimension $2$. If $\n$ is indecomposable, from \cite[section 6.2]{T-W}, $Z(\n)\subseteq \n^2$, but using (\ref{fdimension}), $Z(\n)$ must be $3$-dimensional, which implies $\mathrm{dim}\, \n\geq 6$, a contradiction.

\end{proof}


\begin{theo}\label{sym3} Over any field $\mathbb{K}$ of characteristic zero, the categories ${\bf Sym_0(3,2)}$ and ${\bf Sym_0(3,3)}$ satisfy the following features:

\begin{enumerate}
\item[{\rm (i)}]The objects in ${\bf Sym_0(3,2)}$ are the invariant forms $B_{3,2}^{A;\gamma}$ described in Example \ref{formasinvariantespequenas2} and given by any $3\times 3$ symmetric matrix $A$ and any $0\neq \gamma\in \mathbb{K}$. Moreover, any $B_{3,2}^{A;\gamma}$ is isomorphic to $B_{3,2}^{\mathbf{0}_{3\times 3};1}$ in ${\bf Sym_0(3,2)}$.

\item[{\rm (ii)}] The objects $B\in {\bf Sym_0(3,3)}$ such that $\mathbf{Q_{3,3}}(B)$ is an indecomposable quadratic Lie algebra are the invariant forms $B_{3,3}^{A_1;\gamma;A_2}$ described in \emph{Example \ref{formasinvariantespequenas2}},  and given by any scalar $\gamma\in \mathbb{K}$ and any pair of $3\times 3$ symmetric matrices ($A_1$,$A_2$) with ${\rm rank\, A_2}\geq 2$. Moreover, for symmetric matrices $A_2,B_2$ of ${\rm rank}\, \geq 2$, $B_{3,3}^{A_1;\delta;A_2}$ and $B_{3,3}^{B_1;\gamma;B_2}$ are isomorphic in ${\bf Sym_0(3,3)}$ if and only if $CB_2C=({\rm Adj}\, P)^tCA_2C({\rm Adj}\, P)$ for some $3\times 3$ regular matrix $P$ and $C$ the matrix given in \emph{(\ref{Wgamma-C})}.
\end{enumerate}
\end{theo}

\begin{proof} The first assertion in (i) follows from Lemma \ref{restricciones}. Now, let  $A_1={\small\left(
\begin{array}{ccc}
p& q& r\\
q& s&t\\
r& t&u
\end{array}
\right)}$ and $\gamma\neq 0$
be and consider $\theta_P\in \text{Aut}\,\n_{2,3}$ the automorphism given by the matrix (in the Hall basis $\mathcal{H}_{3,2}$):$$P=\left( \begin{array} {ccc|ccc}\gamma & 0 & 0 & 0 & 0 & 0\\ 0 & 1 & 0 & 0 & 0& 0\\ 0& 0 & 1 & 0 & 0 & 0\\ \hline  0 & 0 &\frac{u}{2} & \gamma & 0 & 0 \\ 0 &-\frac{s}{2} & -t  & 0 & \gamma & 0\\ \frac{p}{2\gamma} & \frac{q}{\gamma}  & \frac{r}{\gamma} &  0 & 0 & 1
\end{array}\right).
$$We have that  $P^tB_{3,2}^{\mathbf{0}_{3\times 3};1}P=B_{3,2}^{A_1
;\gamma}
$, which proves the second assertion in (ii). 

Let us assume now that $B_{3,3}^{A_1;\gamma;A_2}\in Obj ({\bf Sym_0(3,3)})$. So the matrix of this form is as in (\ref{A1A2}) and (\ref{A2A2prima}). 
The elements in  the automorphism group ${\rm Aut}\, \mathfrak{n}_{3,3}$ with respect to the Hall basis $\mathcal{H}_{3,3}$ have the matrix form (the $3\times 3$ matrix $C$ is defined in (\ref{Wgamma-C})):$$
X(P;U,V)=\left( \begin{array} {c|c|c} P& 0 & 0\\\hline U&C({\rm Adj}\, P)C&0\\\hline V&U^\prime&P^{\prime\prime} \end{array}\right)
$$where $P\in GL(3)$, $U\in M_{3\times 3}(\mathbb{K}), V\in M_{8\times 3}(\mathbb{K})$, $P^{\prime\prime}\in M_{8\times 8}$ is completely determined by $P$ and $ U^\prime \in M_{8\times 3}(\mathbb{K})$ is related to $U$ as follows$$U={\small\left(
\begin{array}{ccc}
x_1& x_2& x_3\\
x_4& x_5&x_6\\
x_7& x_8&x_9
\end{array}
\right)},U'^t={\small\left(
\begin{array}{cccccccc}
x_2&-x_1&-x_8&x_5&x_8-x_4&0&-x_7&0\\
x_3&0&-x_9-x_1&x_6&x_9&-x_4&0&-x_7\\
0&x_3&-x_2&0&x_6&-x_5&x_9&-x_8
\end{array}
\right)}.$$
The automorphisms $\theta_{X(P;0,0)}\in \mathrm{Aut}\, (\n_{3,3})$ given by the matrix $X(P;0,0)$ are those corresponding to the subgroup $H(3,3)$. For any arbitrary $\theta_{X(P;0,0)}$, the matrix of $(B_{3,3}^{A_1;\gamma;A_2})_{\theta_{X(P;0,0)}}$ is
$$ \left ( \begin{array} {c|c|c} P^t& 0 & 0 \\ \hline 0 & C({\rm Adj}\, P^t)C& 0 \\ \hline 0 & 0 & P^{\prime\prime t}\end{array}\right) 
\left ( \begin{array} {c|c|c}A_1& W_\gamma& A_2^\prime\\ \hline W_\gamma& A_2& 0 \\ \hline A_2^{\prime t} & 0 & 0\end{array} \right)
\left ( \begin{array} {c|c|c} P& 0 & 0 \\ \hline 0 & C({\rm Adj}\, P)C& 0 \\ \hline 0 & 0 & P^{\prime\prime}\end{array}\right)=$$
 $$\left ( \begin{array} {c|c|c}P^tA_1P & W_{\gamma{\rm det}\, P}& P^tA_2^\prime P^{\prime\prime} \\ \hline  W_{\gamma{\rm det}\, P} & C({\rm Adj}\, P^t)CA_2C({\rm Adj}\, P)C& 0 \\ \hline P^{\prime\prime t}A_2^{\prime t}P& 0 & 0 \end{array} \right).$$

The automorphisms $\theta_{X(I_3;U,V)}$ belong to the subgroup $N(3,3)$. For any arbitrary $\theta_{X(I_3;U,V)}$, the matrix of the bilinear form $(B_{3,3}^{A_1;\gamma;A_2})_{\theta_{X(I_3;U,V)}}$ is
$$ \left ( \begin{array} {c|c|c} I_3& U^t & V^t \\ \hline 0 & I_3& U^{\prime t} \\ \hline 0 & 0 & I_8\end{array}\right) 
\left ( \begin{array} {c|c|c}A_1& W_\gamma& A_2^\prime\\ \hline W_\gamma & A_2& 0 \\ \hline A_2^{\prime t} & 0 & 0\end{array} \right)
\left ( \begin{array} {c|c|c} I_3& 0 & 0 \\ \hline U & I_3& 0 \\ \hline V & U^\prime & I_8\end{array}\right)=$$
 $$\left ( \begin{array} {c|c|c} A_1+U^tW_\gamma+V^tA_2^{\prime t} +W_\gamma U+U^tA_2U+A_2^\prime V & W_\gamma+U^tA_2+A_2^ \prime U^\prime & A_2^\prime \\ \hline W_\gamma+A_2U+U^{\prime t} A_2^{\prime t} & A_2& 0 \\ \hline A_2^{\prime t}& 0 & 0 \end{array} \right).$$
The matrix equation in $U$,
\begin{equation}\label{redu2}
W_\gamma+U^tA_2+A_2^\prime U^{\prime } ={\bf 0}_{3\times 3}
\end{equation}
can be viewed as a homogeneous system of $9$ equations in the variables $x_1,\dots, x_9$, that 
that in fact are always repeated the following equation:
\begin{equation}\label{redu1}
\gamma+cx_1-bx_2+ax_3+ex_4-dx_5+bx_6+fx_7-ex_8+cx_9=0.
\end{equation}
Since $A_2\neq 0$, it is immediate that, for any $\gamma\in \mathbb{K}$, we can find adequate entries for $U$ so that the identity (\ref{redu1}) follows. Hence, any form $B_{3,3}^{A_1;\gamma;A_2}$ is isomorphic to a form $B_{3,3}^{Y;0;A_2}$ for some symmetric $3\times 3$ matrix $Y$.

We also claim that, if ${\rm rank}\, A_2\geq 2$, for any $3\times 3$ symmetric matrix $Y$ there exits an automorphism $\varphi \in N(3,3)$ such that $(B_{3,3}^{{\bf 0}_{3\times 3};0;A_2})_\varphi=B_{3,3}^{Y;0;A_2}$. This assertion is equivalent to find $U\in M_{3\times 3}(\mathbb{K}), V\in M_{8\times 3}(\mathbb{K})$ that  are solutions to the system of matrix equations:

\begin{equation}
{\bf 0}_{3,3}=U^t A_2+A_2'U^{\prime }\quad {\rm and}\quad Y=V^tA_2^{\prime t}+U^tA_2U+A_2^\prime V
\end{equation}

From the above reasoning on the equation (\ref{redu2}), for $\gamma=0$ there exist a $3\times 3$ matrix $U_0$ such that ${\bf 0}_{3,3}=U_0^t A_2+A_2'(U_0)^{\prime}$. Since $Y-U_0^tA_2U_0$ is symmetric, the equation  
\begin{equation}
Y-U_0^tA_2U_0=V^tA_2^{\prime t}+A_2^\prime V
\end{equation}
is reduced to find a matrix $V\in M_{8\times 3}(\mathbb{K})$ such that $V^tA_2^{\prime t}+A_2^\prime V$ is any symmetric $3\times 3$ matrix. This is equivalent to find a $8\times 3$ matrix $V$ such that $A_2^\prime V$ is any $3\times 3$ matrix. For a $3\times 1$ vector $v$, the first column $X$ of the matrix $V$ must satisfy $A_2^\prime X=v$. But this equation has solution because the rank of $A_2^\prime$ is $3$ according to Lemma {restricciones}. In the same vein, we can obtain the second and third columns of $V$ and so $A_2^\prime V$ is any $3\times 3$ matrix. This proves our claim.

Now we will prove the equivalence in (ii). Let $B_{3,3}^{A_1;\delta;A_2}$ and $B_{3,3}^{B_1;\gamma;B_2}$ be isomorphic objects in ${\bf Sym_0(3,3)}$. Then,  there exists $\tau\in \text{Aut}\,\n_{2,3}=H(3,3)N(3,3)$ such that $(B_{3,3}^{A_1;\delta;A_2})_\tau=B_{3,3}^{B_1;\gamma;B_2}$. From the previous reasoning, we can assume that $\delta=\gamma=0$. Decomposing, $\tau=\sigma\varphi$ with $\sigma\in H(3,3)$ and $\varphi\in N(3,3)$, we have $B_{3,3}^{B_1;0;B_2}=((B_{3,3}^{A_1;0;A_2})_\sigma)_\varphi$ and therefore $B_2=C({\rm Adj}\, P)^tCA_2C({\rm Adj}\, P)C$.

For the converse, suppose  $B_2=C({\rm Adj}\, P)^tCA_2C({\rm Adj}\, P)C$ and ${\rm rank}\, A_2\geq 2$. From the automorphism $\sigma\in H(3,3)$ given by the matrix $X(P;0,0)$,
$$(B_{3,3}^{A_1;0;A_2})_\sigma=B_{3,3}^{P^tA_1P;0;B_2}.
$$
As ${\rm rank}\, B_2={\rm rank}\, A_2\geq 2$, there exist $\varphi_1, \varphi_2 \in N(3,3)$ such that $((B_{3,3}^{P^tA_1P;0;B_2})_{\varphi_1})_{\varphi_2}=(B_{3,3}^{0;0;B_2})_{\varphi_2}=B_{3,3}^{B_1;0;B_2}$.  Hence $(B_{3,3}^{A_1;0;A_2})_{\sigma\varphi_1\varphi_2}=B_{3,3}^{B_1;0;B_2}$, which proves the isomorphism.
\end{proof}

\begin{rem} In item (ii) of Theorem \ref{sym3}, if ${\rm rank\, A_2}=1$, the equality  $CB_2C=({\rm Adj}\, P)^tCA_2C({\rm Adj}\, P)$ is a necessary condition of isomorphism, but it is not sufficient. The invariant forms $B_{3,3}^{A_1;0;A_2}$ and $B_{3,3}^{{\mathbf 0};0;A_2}$  with $$
A_1={\small\left(
\begin{array}{ccc}
0& 0& 0\\
0& 0&0\\
0& 0&1
\end{array}
\right)}\qquad {\rm and}\qquad A_2={\small\left(
\begin{array}{ccc}
1& 0& 0\\
0& 0&0\\
0& 0&0
\end{array}
\right)}$$+are not isomorphic elements in ${\bf Sym_0(3,3)}$
\end{rem}

\begin{coro}\label{type3} Over any field ${\mathbb K}$ of characteristic $0$, up to metric isomorphism the indecomposable quadratic Lie algebras of type $3$ and nilpotent indices $2$ or $3$ are:

\begin{enumerate}

\item[{\rm (i)}] The quadratic Lie algebra $(\mathfrak{n}_{3,2},B_{3,2}^{\mathbf{0}_{3\times 3};1})$.

\item[{\rm (ii)}] The quadratic Lie algebras $(\frac{\mathfrak{n}_{3,3}}{{\rm Ker}\, B_{3,3}^{0;0;A}}, \overline{B_{3,3}^{0;0;A}})$ for any $3\times 3$ symmetric matrix $A$ of rank $\geq 2$. Moreover, $(\frac{\mathfrak{n}_{3,3}}{{\rm Ker}\, B_{3,3}^{0;0;A}}, \overline{B_{3,3}^{0;0;A}})$ and $(\frac{\mathfrak{n}_{3,3}}{{\rm Ker}\, B_{3,3}^{0;0;B}}, \overline{B_{3,3}^{0;0;B}})$ are isometrically isomorphic if and only if $CBC=({\rm Adj}\, P)^tCAC({\rm Adj}\, P)$ for some $3\times 3$ regular matrix $P$. 
\end{enumerate}
In (ii),  if $\mathbb {K}^2= \mathbb {K}$ or $\mathbb {K}=\mathbb {R}$, then $(\frac{\mathfrak{n}_{3,3}}{{\rm Ker}\, B_{3,3}^{0;0;A}}, \overline{B_{3,3}^{0;0;A}})$ and $(\frac{\mathfrak{n}_{3,3}}{{\rm Ker}\, B_{3,3}^{0;0;B}}, \overline{B_{3,3}^{0;0;B}})$ are isometrically isomorphic if and only if $A$ and $B$ are congruent. 
\end{coro}

\begin{proof} We just need to apply the previous Therorem \ref{sym3} and Corollary \ref{isomorfismo}
to prove (i) and  (ii). For the last statement, since $C^t=C$ we only need to show that if $A$ and $B$ are congruent then $(\frac{\mathfrak{n}_{3,3}}{{\rm Ker}\, B_{3,3}^{0;0;A}}, \overline{B_{3,3}^{0;0;A}})$ and $(\frac{\mathfrak{n}_{3,3}}{{\rm Ker}\, B_{3,3}^{0;0;B}}, \overline{B_{3,3}^{0;0;B}})$ are isometrically isomorphic. First we notice that if $D$ is a  regular matrix, $ Adj (D)=det (D)(D^{-1})^t $. We apply this now. We suppose that $B= R^tAR$, and we want to find a regular matrix $P$ such that $R= CAdj(P)C$. Since $C=C^{-1}$ this implies that $Adj(P)= CRC$. So if we find a matrix $P$ such that $Adj(P)=CRC$, then we have finished.  In the case $\mathbb {K}^2= \mathbb {K}$, if  we take $P=\sqrt {det(R)} C(R^{-1})^t C$,  we can check that $Adj(P)=CRC$. Indeed, as we have shown $ Adj (P)=det (P)(P^{-1})^t $ and then
$$Adj(P)= (\sqrt {det(R)})^3 det (R)^{-1} (\sqrt{det (R)})^{-1} CRC=CRC$$
In the case  $\mathbb {K}= \mathbb{R}$, we observe that if $R$ is such that $det (R)<0$ we can consider $S=-R$ and then also $S^tA S=B$ and $det (S)>0$. So we can suppose that $det (R)>0$ and then taking $P=\sqrt {det(R)} C(R^{-1})^t C$,  we have also $Adj(P)=CRC$.
\end{proof}

\section{$\mathbb{K}$ algebraically closed and $\mathbb{K} = \mathbb{R}$}

In this final section, first we restrict ourshelves to algebraically closed fields to reach, up to isomorphism, the nilpotent quadratic Lie algebras on $2$ and $3$ generators and small nilpotent index. In this case, we note that for a given $B\in GL(3)$, $B=Adj(A)$ with $A=\sqrt{{\rm det}\, B}(B^{-1})^t$. 

\begin{theo} \label{teoremafinal1} Up to isomorphism, the indecomposable quadratic nilpotent Lie algebras over any algebraically closed filed $\mathbb{K}$ of characteristic zero ${\mathbb K}$ of type 1 or of type 2 and nilindex $\leq 5$ or of  type 3 and nilindex $\leq 3$ are isomorphic to one of the following Lie algebras (where $\phi, \varphi, \varphi_i, \psi $ and $\psi_i$ are symmetric bilinear forms):
\begin{enumerate}
\item[\emph{(i)}] The $1$-dimensional abelian Lie algebra $(\mathfrak{n}_{1,1}, \phi)$  with basis $\{ a_1\}$ and $\phi(a_1,a_1)=1$.

\item[\emph{(ii)}] The $5$-dimensional free nilpotent Lie algebra $(\mathfrak{n}_{2,3},\varphi)$ with basis $\{ a_i\}_{i=1}^5$ and nonzero products $[a_2,a_1]=-[a_1,a_2]=a_3$, $[a_3,a_1]=-[a_1,a_3]=a_4$ and $[a_3,a_2]=-[a_2,a_3]=a_5$ where $\varphi(a_i,a_j)=(-1)^{i-1}$ for $i\leq j$ and $i+j=6$ and $\varphi(a_i,a_j)=0$ otherwise.

\item[\emph{(iii)}] The $7$-dimensional Lie algebra $(\mathfrak{n}_{2,5}^1, \varphi_1)$ with basis $\{ a_i\}_{i=1}^7$ and nonzero products  $[a_2,a_1]=-[a_1,a_2]=a_3$, $[a_3,a_1]=-[a_1,a_3]=a_4$, $[a_4,a_1]=-[a_1,a_4]=a_5$, $[a_5,a_1]=-[a_1,a_5]=a_6$, $[a_5,a_2]=-[a_2,a_5]=a_7$ and $[a_3,a_4]=-[a_4,a_3]=a_7$ where $\varphi_1(a_i,a_j)=(-1)^i$ for $i\leq j$ and $i+j=8$ and $\varphi_1(a_i,a_j)=0$ otherwise.

\item[\emph{(iv)}] The $8$-dimensional Lie algebra $(\mathfrak{n}_{2,5}^2, \varphi_2)$ with basis $\{ a_i\}_{i=1}^8$ and nonzero products $[a_2,a_1]=-[a_1,a_2]=a_3$, $[a_3,a_1]=-[a_1,a_3]=a_4$,$[a_3,a_2]=-[a_2,a_3]=a_5$ $[a_4,a_1]=-[a_1,a_4]=a_6$, $[a_6,a_1]=-[a_1,a_6]=a_7$, $[a_6,a_2]=-[a_2,a_6]=a_8$, $[a_2,a_5]=-[a_5,a_2]=-a_6, [a_4,a_3]=-[a_3,a_4]=-a_8,$ and $[a_5,a_3]=-[a_3,a_5]=a_7$,  where $\varphi_2(a_i,a_j)=(-1)^i$ for $0\leq i\leq 3$ and $ i+j=9$, $\varphi_2 (a_4,a_4)= \varphi_2(a_5,a_5)=1$ and $\varphi_2(a_i,a_j)=0$ otherwise.

\item[\emph{(v)}] The $6$-dimensional free nilpotent Lie algebra $(\mathfrak{n}_{3,2},\psi)$ with basis $\{ a_i\}_{i=1}^6$ and nonzero products $[a_2,a_1]=-[a_1,a_2]=a_4$, $[a_3,a_1]=-[a_1,a_3]=a_5$ and $[a_3,a_2]=-[a_2,a_3]=a_6$ where $\psi(a_i,a_j)=(-1)^{i-1}$ for $i\leq j$ and $i+j=7$ and $\psi(a_i,a_j)=0$ otherwise.

\item[\emph{(vi)}] The $8$-dimensional Lie algebra $(\mathfrak{n}_{3,3}^1,\psi_1)$ with basis $\{ a_i\}_{i=1}^8$ and nonzero products $[a_2,a_1]=-[a_1,a_2]=a_4$, $[a_3,a_1]=-[a_1,a_3]=a_5$, $[a_4,a_1]=-[a_1,a_4]=a_6$, $[a_4,a_2]=-[a_2,a_4]=a_7$, $[a_5,a_1]=-[a_1,a_5]=a_8$ and $[a_5,a_3]=-[a_3,a_5]=a_7$ where $\psi_1(a_4,a_4)=\psi_1(a_5,a_5)=1$, $\psi_1(a_1,a_7)=1=-\psi_1(a_2,a_6)=-\psi_1(a_3,a_8)$ and $\psi_1(a_i,a_j)=0$ otherwise.

\item[\emph{(vii)}] The $9$-dimensional Lie algebra $(\mathfrak{n}_{3,3}^2,\psi_2)$ with basis $\{ a_i\}_{i=1}^9$ and nonzero products $[a_2,a_1]=-[a_1,a_2]=a_4$, $[a_3,a_1]=-[a_1,a_3]=a_5$ $[a_3,a_2]=-[a_2,a_3]=a_6$, $[a_4,a_1]=-[a_1,a_4]=a_7$, $[a_4,a_2]=-[a_2,a_4]=a_8$, $[a_5,a_1]=-[a_1,a_5]=a_9$, $[a_5,a_3]=-[a_3,a_5]=a_8, [a_6,a_3]=-[a_3,a_6]=-a_7$ and $[a_6,a_2]=-[a_2,a_6]=a_9$ where $\psi_2(a_4,a_4)=\psi_2(a_5,a_5)=\psi_2(a_6,a_6)=1$, $\psi_2(a_1,a_8)=1=-\psi_2(a_2,a_7)=-\psi_2(a_3,a_9)$ and $\psi_2(a_i,a_j)=0$ otherwise.
\end{enumerate}

\noindent Any nonabelian  quadratic Lie algebra of type $\leq 2$ is indecomposable. The nonabelian decomposable Lie algebras of type $3$ and nilpotent index $\leq 3$ are the orthogonal sum $(\mathfrak{n}_{1,1}, \phi)\oplus L$ where $L$ is a quadratic Lie algebra as in item \emph{(ii), (iii) or (iv)}.
\end{theo}

\begin{proof} 
Let us assume now that $(\mathfrak{n}, \varphi)$ is an indecomposable nilpotent quadratic nonabelian Lie algebra. If $\n$ has type 2, according to Corollary \ref{type2}, either $(\mathfrak{n}, \varphi)$ is as in item (ii) or $(\mathfrak{n}, \varphi)=(\frac{\mathfrak{n}_{2,5}}{{\rm Ker}\, B_{2,5}^{0;0;A}},\overline{B_{2,5}^{0;0;A}})$ where, up to isometric isomorphism, $A$ is one of the following matrices:
$$A_1=\begin{pmatrix}1&0\\0&0\end{pmatrix}\quad {\rm or}\quad A_2=\begin{pmatrix}1&0\\0&1\end{pmatrix}$$and $B_{2,5}^{0;0;A}$ is as described in Example \ref{formasinvariantespequenas}. The rank and the dimension of the kernel  of the form ${\rm B_{2,5}^{0;0;A_1}}$ are both $7$, and moreover
$$
{\rm Ker}\, B_{2,5}^{0;0;A_1}=span<[[x_2,x_1],x_2],[[[x_2,x_1],x_1],x_2],[[[x_2,x_1],x_2],x_2],$$
$$[[[x_2,x_1],x_1],x_1],x_2]+[[[x_2,x_1],x_1],[x_2,x_1]], [[[[x_2,x_1],x_1],x_2],x_2],[[[x_2,x_1],x_2],[x_2,x_1]],$$ $$[[[[x_2,x_1],x_2],x_2],x_2]>.
$$
The rank and the dimension of the kernel of the form ${\rm B_{2,5}^{0;0;A_2}}$ are $8$ and $6$, respectively, and
$$
{\rm Ker}\, B_{2,5}^{0;0;A_2}=span<[[[x_2,x_1],x_1],x_2],[[[x_2,x_1],x_2],x_2]-[[[x_2,x_1],x_1],x_1],$$ $$[[[x_2,x_1],x_1],[x_2,x_1]]+[[[x_2,x_1],x_1],x_1],x_2], [[[[x_2,x_1],x_1],x_2],x_2],$$ $$[[[x_2,x_1],x_2],[x_2,x_1]]-[[[[x_2,x_1],x_1],x_1],x_1],[[[[x_2,x_1],x_2],x_2],x_2]-[[[[x_2,x_1],x_1],x_1],x_2]>.
$$In this way we obtain the quadratic algebras in items (iii) and (iv).

If $\n$ has type 3, from Corollary \ref{type3} either $(\mathfrak{n}, \varphi)$ is as in item (v) or $(\mathfrak{n}, \varphi)=(\frac{\mathfrak{n}_{3,3}}{{\rm Ker}\, B_{3,3}^{0;0;A}},\overline{B_{3,3}^{0;0;A}})$ where, up to isometric isomorphism, $A$ is one of the following matrices:
$$B_1=\begin{pmatrix}1&0&0\\0&1&0\\0&0&0\end{pmatrix}\quad {\rm or}\quad B_2=\begin{pmatrix}1&0&0\\0&1&0\\0&0&1\end{pmatrix},$$and $B_{3,3}^{0;0;A}$ is described in Example \ref{formasinvariantespequenas2}. The rank and the dimension of the kernel  of the form ${\rm B_{3,3}^{0;0;B_1}}$ are $8$ and $6$ respectively, and
$$
{\rm Ker}\, B_{3,3}^{0;0;B_1}=span<[x_3,x_2],[[x_2,x_1],x_3],[[x_3,x_1],x_2],$$ $$[[x_3,x_1],x_3]-[[x_2,x_1],x_2],
[[x_3,x_2],x_2],[[x_3,x_2],x_3]>.
$$
The rank and the dimension of the form ${\rm B_{3,3}^{0;0;B_2}}$ are $9$ and $5$ respectively, and
$$
{\rm Ker}\, B_{3,3}^{0;0;B_2}=span<[[x_2,x_1],x_3],[[x_3,x_1],x_2],[[x_3,x_1],x_3]-[[x_2,x_1],x_2],$$ $$[[x_3,x_2],x_2]-[[x_3,x_1],x_1],[[x_3,x_2],x_3]-+x_2,x_1],x_1]>$$
So we get the quadratic algebras in items (vi) and (vii). The final assertion is straightforeword.
\end{proof}

\bigskip

Now, we consider the case $\mathbb{K}=\mathbb{R}$. 

\begin{theo} Up to isomorphism, the indecomposable quadratic nilpotent real Lie algebras of type 1 or of type 2 and nilindex $\leq 5$ or of type 3 and nilindex $\leq 3$ are the following (where $\varphi_3, \psi_3$ and $\psi_4$ are symmetric bilinear forms):
\begin{enumerate}
\item[\emph{(i)}] The quadratic Lie algebras $(\mathfrak{n}, \pm \varphi)$ where $\mathfrak{n}$ and $\varphi$ is as described in items \emph{(i), (ii), (iii), (iv), (vi), (vii)} of Theorem \ref{teoremafinal1} and $(\mathfrak{n}_{3,2}, \psi)$ as in item \emph{(v)}.

\item[\emph{(ii)}] The $8$-dimensional Lie algebra $(\mathfrak{n}_{2,5}^3, \varphi_3)$ with basis $\{ a_i\}_{i=1}^8$ and nonzero products $[a_2,a_1]=-[a_1,a_2]=a_3$, $[a_3,a_1]=-[a_1,a_3]=a_4$, $[a_4,a_1]=-[a_1,a_4]=a_6$, $[a_6,a_1]=-[a_1,a_6]=a_7$,  ,$[a_3,a_2]=-[a_2,a_3]=a_5$, ,$[a_5,a_2]=-[a_2,a_5]=-a_6$, $[a_6,a_2]=-[a_2,a_6]=a_8$, ,$[a_4,a_3]=-[a_3,a_4]=-a_8$ and ,$[a_5,a_3]=-[a_3,a_5]=-a_7$ where $\varphi_3(a_i,a_j)=(-1)^i$ for $i\leq j$, $i=1,2,3$ and $i+j=9$, $\varphi_3(a_4,a_4)=1=-\varphi_3(a_5,a_5)$ and $\varphi_3(a_i,a_j)=0$ otherwise.

\item[\emph{(iii)}] The $8$-dimensional Lie algebra $(\mathfrak{n}_{3,3}^3,\psi_3)$ with basis $\{ a_i\}_{i=1}^8$ and nonzero products $[a_2,a_1]=-[a_1,a_2]=a_4$, $[a_3,a_1]=-[a_1,a_3]=a_5$, $[a_4,a_1]=-[a_1,a_4]=a_6$, $[a_5,a_1]=-[a_1,a_5]=a_8$, $[a_4,a_2]=-[a_2,a_4]=a_7$ and $[a_5,a_3]=-[a_3,a_5]=-a_7$ where $\psi_3(a_4,a_4)=-\psi_3(a_5,a_5)=1$, $\psi_3(a_1,a_7)=1=-\psi_3(a_2,a_6)=\psi_3(a_3,a_8)$ and $\psi_3(a_i,a_j)=0$ otherwise.

\item[\emph{(iv)}] The $9$-dimensional quadratic Lie algebras $(\mathfrak{n}_{3,3}^4,\pm \psi_4)$ with basis $\{ a_i\}_{i=1}^9$ and nonzero products $[a_2,a_1]=-[a_1,a_2]=a_4$, $[a_3,a_1]=-[a_1,a_3]=a_5$, $[a_4,a_1]=-[a_1,a_4]=a_7$, $[a_5,a_1]=-[a_1,a_5]=a_9$ $[a_3,a_2]=-[a_2,a_3]=a_6$, $[a_4,a_2]=-[a_2,a_4]=a_8$, $[a_6,a_2]=-[a_2,a_6]=-a_9$, $[a_5,a_3]=-[a_3,a_5]=a_8$  and $[a_6,a_3]=-[a_3,a_6]=a_7$, where $\psi_4(a_4,a_4)=\psi_4(a_5,a_5)=1=-\psi_4(a_6,a_6)$, $\psi_4(a_1,a_8)1=-\psi_4(a_2,a_7)=-\psi_4(a_3,a_9)$ and $\psi_2(a_i,a_j)=0$ otherwise.
\end{enumerate}

\noindent Any nonabelian  quadratic Lie algebra of type $\leq 2$ is indecomposable. The nonabelian decomposable Lie algebras of type $3$ and nilpotent index $\leq 3$ are the orthogonal sum $(\mathfrak{n}_{1,1}, \phi)\oplus L$ where $L$ is a quadratic Lie algebra as in items \emph {(ii), (iii) or (iv)}.
\end{theo}

\begin{proof} According to Corollaries \ref{type2} and \ref{type3}, over the real field the classification is given by the rank and the signature of $n\times n$ symmetric matrices for $n=1,2,3$. So, the cases to be considered are:
\begin{itemize}
\item[$\bullet$] $(\mathfrak{n}_{1,1}, \pm \phi)$ with basis $\{ a_1\}$ and $\phi (a_1,a_1)=1$ and $(\mathfrak{n}_{3,2},  \psi)$ as described in item (v) of Therorem \ref{teoremafinal1}.
\item[$\bullet$] $(\mathfrak{n}, \pm\varphi)$ where $(\mathfrak{n}, \pm\varphi)$ is as in item (ii) or  Theorem 6.1, $(\mathfrak{n}, \varphi)=(\frac{\mathfrak{n}_{2,5}}{{\rm Ker}\, B_{2,5}^{0;0;A}},\overline{B_{2,5}^{0;0;A_i}})$ where, up to isometric isomorphism, $A_i$ is one of the following matrices:
$$\pm A_1=\begin{pmatrix}1&0\\0&0\end{pmatrix}, \pm A_2=\begin{pmatrix}1&0\\0&1\end{pmatrix}\quad {\rm or}\quad A_3=\begin{pmatrix}1&0\\0&-1\end{pmatrix}$$and $B_{2,5}^{0;0;A_i}$ is as described in Example \ref{formasinvariantespequenas}.

\item[$\bullet$]$(\mathfrak{n}, \varphi)=(\frac{\mathfrak{n}_{3,3}}{{\rm Ker}\, B_{3,3}^{0;0;A}},\overline{B_{3,3}^{0;0;C_i}})$ where, up to isometric isomorphism, $C_i$ is one of the following matrices:
$$\pm C_1=\begin{pmatrix}1&0&0\\0&1&0\\0&0&0\end{pmatrix}, C_2=\begin{pmatrix}1&0&0\\0&-1&0\\0&0&0\end{pmatrix}, \pm C_3=\begin{pmatrix}1&0&0\\0&1&0\\0&0&1\end{pmatrix}, \pm C_4=\begin{pmatrix}1&0&0\\0&1&0\\0&0&-1\end{pmatrix}$$and $B_{3,3}^{0;0;C_i}$ is as described in Example \ref{formasinvariantespequenas2}. 
\end{itemize}

The matrices $\pm A_1, \pm A_2, \pm C_1, \pm C_3$, and $(\mathfrak{n}_{1,1}, \pm \phi)$, $(\mathfrak{n}_{2,3}, \pm \phi)$ and $(\mathfrak{n}_{3,2}, \psi)$ yield to the quadratic algebras in item (i) of this theorem. A straightforward calculation gives us  the rank and the dimension of the kernel of the form ${\rm B_{2,5}^{0;0;A_3}}$, which are $8$ and $6$ respectively, and
$$
{\rm Ker}\, B_{2,5}^{0;0;A_3}=span<[[[x_2,x_1],x_1],x_2],[[[x_2,x_1],x_1],x_1]+[[[x_2,x_1],x_2],x_2],
$$
$$[[[[x_2,x_1],x_1],x_1],x_2]+[[[x_2,x_1],x_1],[x_2,x_1]], [[[[x_2,x_1],x_1],x_2],x_2], [[[[x_2,x_1],x_1],x_1],x_1]+
$$
$$[[[x_2,x_1],x_2],[x_2,x_1]],[[[[x_2,x_1],x_1],x_1],x_2]+[[[[x_2,x_1],x_2],x_2],x_2]>.
$$The corresponding quotient provides the quadratic algebra described in (ii). The rank and the dimension of the kernel  of the forms ${\rm B_{3,3}^{0;0;C_2}}$ and ${\rm B_{3,3}^{0;0;C_4}}$ are $8$ and $6$ and $9$ and $5$ respectively. Moreover, 
$$
{\rm Ker}\, B_{3,3}^{0;0;C_2}=span<[x_3,x_2],[[x_2,x_1],x_3],[[x_3,x_1],x_2],[[x_2,x_1],x_2]+[[x_3,x_1],x_3],$$
$$[[x_3,x_2],x_2],[[x_3,x_2],x_3]>.
$$and
$$
{\rm Ker}\, B_{3,3}^{0;0;\pm C_4}=span<[[x_2,x_1],x_3],[[x_3,x_1],x_2],-[[x_2,x_1],x_2]+[[x_3,x_1],x_3],$$
$$[[x_3,x_1],x_1]+[[x_3,x_2],x_2],-[[x_2,x_1],x_1]+[[x_3,x_2],x_3]>.
$$

The corresponding two quotients give the quadratic algebras in items (iii) and (iv).
\end{proof}

\bigskip

\begin{rem} 
If we consider the general problem of building a $t$-nilpotent quadratic Lie algebra with $d$ generators beginning with a $B(e_i,e_j)$ with $i+j=t+1$ in $\mathfrak{n}_{d,t}$, one can think that this seems easy, but one has really to check some facts about the final $B$. It takes some work to prove what is needed in a particular case. In fact, we do not know what is necessary in general. This can be a line of research to address in the future. 
\end{rem}

\bigskip

\centerline{ACKNOWLEDGEMENTS }

\bigskip

The authors thank the referee for very valuable suggestions that have improved the paper.

{\noindent
\author{Pilar Benito, Daniel de-la-Concepci\'on, Jes\'us Laliena \footnote {The authors have been
supported by the Spanish Ministerio de Econom\'{\i}a y Competitividad (MTM 2013-45588-CO3-3), and the second author also by a FPU grant (12/03224) of the Spanish Ministerio de Educaci\'on y Ciencia} }
\\{\small Departamento de Matem\'aticas y Computaci\'on}\\
{\small  Universidad de La Rioja}\\
{\small  26004, Logro\~no. Spain}\\
{\small pilar.benito@unirioja.es }, {\small daniel-de-la.concepcion@unirioja.es},  {\small jesus.laliena@unirioja.es}

\end{document}